\let\ORIlabel\label
\let\ORIrefstepcounter\refstepcounter
\AddToHook{package/hyperref/before}{
    \let\label\ORIlabel
    \let\refstepcounter\ORIrefstepcounter
}
\documentclass[final,onefignum,onetabnum]{siamart220329}
\usepackage{amsfonts,amsmath,mathrsfs,bm} %,mathtools}
\usepackage{amssymb}
\usepackage[notref,notcite]{showkeys}
\usepackage{graphicx}
\usepackage{wrapfig}
\usepackage{subfig}
\usepackage{float}
\usepackage{bbm}
\usepackage{algorithmic}
\usepackage{gensymb,color,soul}
\usepackage{tikz}
\usepackage{tikz-3dplot}
\usepackage{booktabs}

\newcommand{\R}{\mathbb{R}}

\newcommand{\Acal}{\mathcal{A}}

\newtheorem{remark}[theorem]{Remark}
\newtheorem{own_algorithm}{Algorithm}

\DeclareMathOperator{\argmin}{\rm arg\, min}

\newcommand{\vertiii}[1]{{\left\vert\kern-0.25ex\left\vert\kern-0.25ex\left\vert #1 
\right\vert\kern-0.25ex\right\vert\kern-0.25ex\right\vert}}

\newtheorem{example}{Example}

\usepackage{array} % PH 27.2.2026

\makeatletter
\newcommand*\bigcdot{\mathpalette\bigcdot@{.5}}
\newcommand*\bigcdot@[2]{\mathbin{\vcenter{\hbox{\scalebox{#2}{$\m@th#1\bullet$}}}}}
\makeatother

\title{A bilinear inverse problem with forward operator inaccuracy applied to neonatal atlas-based diffuse optical tomography}

\author{
A.~Hakula\footnotemark[1]
\and P.~Hirvi\footnotemark[1]
\and N.~Hyv\"onen\footnotemark[1] \footnotemark[2]
}

\begin{document}
\maketitle

\renewcommand{\thefootnote}{\fnsymbol{footnote}}
\footnotetext[1]{Aalto University, Department of Mathematics and Systems Analysis, P.O.~Box 11100, FI-00076 Aalto, Finland (aada.hakula@aalto.fi, pauliina.hirvi@aalto.fi, nuutti.hyvonen@aalto.fi). The work of PH and NH was supported by the Research Council of Finland (Flagship of Advanced Mathematics for Sensing Imaging and Modelling grant 359181). The work of AH was part of the Ministry of Education and Culture’s Doctoral Education Pilot under Decision No. VN/3137/2024-OKM-6 (Doctoral Education Pilot for Mathematics of Sensing, Imaging and Modelling). PH was also supported by the Alfred Kordelin Foundation and the Emil Aaltonen Foundation.}
\footnotetext[2]{Corresponding author.}

\begin{abstract} 
In this work, we assume to have a set of candidate forward operator matrices and suggest principal component analysis for modeling their variation from the mean. We use this principal component representation to adapt a bilinear inverse-problem formulation to forward-operator inaccuracy in neonatal atlas-based diffuse optical tomography and present two optimization algorithms, as well as Gibbs sampling and a version of the Bayesian approximation error method, for approximately solving the resulting problem. We apply the algorithms to account for the inaccuracy that is present in the sensitivity profiles or Jacobian matrices in diffuse optical tomography when an atlas-based model of the head anatomy is used instead of the subject's own anatomical model in neonates over a wide range of gestational ages (29--44 weeks). We report visual and numerical improvements in the spatial localization and contrast-to-noise-ratio in reconstructions of simulated hemodynamic activity.  
\end{abstract}

\renewcommand{\thefootnote}{\arabic{footnote}}



\pagestyle{myheadings}
\thispagestyle{plain}
\markboth{A.~HAKULA, P.~HIRVI, AND N.~HYV\"ONEN}{BILINEAR INVERSE PROBLEM WITH OPERATOR INACCURACY IN DOT}

\section{Introduction}
\label{sec:introduction}
Finite- yet high-dimensional linear inverse problems often arise, for example, from discretized models in medical imaging. The forward operator, or matrix in the linear equation is built based on available information on the target. Though in practice it is often unavoidable that the matrix models the system with limited accuracy, including the forward operator inaccuracy to the solution of the inverse problem is not a standard approach in inverse problems. This work tackles such operator inaccuracy in the Bayesian inversion framework assuming certain prior information on the structure of the forward model, with application to {\em diffuse optical tomography} (DOT) under inaccurate information on the head anatomy of the subject. 

A standard deterministic technique for solving linear inverse problems with an inaccurate forward operator is total least squares Tikhonov regularization~\cite{Golub1999}. Although it has connections to some ideas presented in this text, it does not explicitly account for statistical prior information on the inaccuracy in the forward operator. Blind deconvolution is a specific extensively studied application, where the (linear) forward operator is inaccurately known; see, e.g., \cite{Burger01,Chan1998,Justen06,Korolev2018,Perrone2016} and the references therein. Another important category of problems that are closely related to our setting are dynamic inverse problems, where the forward operator changes as a function of time; see, e.g., \cite{Burger2017,Chen2019,Chung2018,Hahn2014,Katsevich2010,USchmitt2002} and the references therein. Motivated by such problems, \cite{Blanke2020} recently considered iterative regularization for problems with inexact forward operators. The Bayesian approximation error method (BAE)~\cite{Kaipio2006}, one of the techniques considered in this work, is a standard statistical approach for handling uncertainties in forward models; see \cite{Arridge_2006,Kolehmainen2009,Mozumder2013,Tarvainen2010,Tarvainen2009} for its applications to DOT. The idea is to effectively reduce the weight of certain directions, or combinations of measurements from the linear model, which carry most of the inaccuracy. Finally, machine learning approaches, particularly deep learning methods, have been increasingly applied also in DOT~\cite{Fry2025,Zou2021}, for example, for post-processing errors due to inaccurate tissue-specific baseline optical parameters~\cite{mozumder2024diffuse}.

\subsection{Bilinear Inverse Problem}
In this work, it is assumed that the system matrix allows an approximate representation as a linear combination of basis matrices, with prior information on the distribution of the associated expansion coefficients, e.g., due to principal component analysis (PCA) on the variations in the measurement setup. Our novelty claim is not the bilinear parameterization itself, which has been studied in imaging and self-calibration problems, but its use for modeling atlas-induced Jacobian uncertainty in neonatal DOT via a PCA basis constructed from registered head-model Jacobians. Including the expansion coefficients as unknowns results in a bilinear inverse problem, which allows for compensating for the uncertainty about the forward model during the solution process, but makes solving the problem more ambiguous. Indeed, even under the simplifying assumption of mutually independent Gaussian priors for the two types of unknowns, the MAP objective for the bilinear inverse problem is biconvex but not jointly convex in general; even the scalar case can possess multiple local minima.

We consider two well-known algorithms for approximating a (local) MAP estimate for our bilinear inverse problem: block coordinate descent (BCD) in the two patches of unknowns and the Gauss--Newton (G--N) method, both of which allow an efficient numerical implementation via the Woodbury matrix formula assuming the number of measurements is relatively low, which is often the case in practice. Although neither algorithm is guaranteed to converge to a global minimum, both converge to local minima in our numerical experiments, with BCD requiring far more iterations, reducing its practical applicability. These algorithms have previously been introduced and tested for (unregularized) bilinear least-squares problems in \cite{Elden2018}, and our BCD implementation was validated on simple geometries in the master's thesis~\cite{AHakulaMT}. In addition to the optimization algorithms, we implement Gibbs sampling from the joint posterior distribution considering both unknowns and, for comparison, apply a variant of the BAE method to approximately marginalize the contribution of the expansion coefficients, reducing the bilinear inverse problem back to a linear one.

In addition to blind deconvolution mentioned above, bilinear or quadratic inverse problems arise in, e.g., deautoconvolution~\cite{Anzengruber16,Gorenflo94}, parallel imaging in magnetic resonance imaging (MRI)~\cite{Blaimer04}, phase retrieval~\cite{Millane90,Seifert06}, random mask imaging~\cite{Bahmani2015}, and self-calibration~\cite{Ling2018}. These works show that bilinear parameterizations of an unknown signal and an unknown sensing or calibration operator are a well-established theme in imaging. In particular, Bahmani and Romberg~\cite{Bahmani2015} studied lifting, identifiability and convex relaxation for blind deconvolution in random mask imaging, while Ling and Strohmer~\cite{Ling2018} considered self-calibration and bilinear inverse problems via linear least squares. Beinert and Bredies introduced a methodology for solving bilinear inverse problems by employing convex relaxation based on tensorial lifting with a rank-one constraint and applying first-order proximal algorithms~\cite{Beinert19,Beinert20}; see also~\cite{Alberti25}.

Our setting differs from these works in that the unknown operator variation is not a generic calibration vector or convolution kernel, but a variation of DOT Jacobian matrices induced by anatomical uncertainty in an atlas of registered neonatal head models. We therefore do not claim novelty for the bilinear inverse-problem formulation itself. Rather, the contribution is the construction and use of a PCA-based Jacobian atlas for neonatal DOT, together with Bayesian optimization, sampling, and approximation-error reconstructions for this application.

\subsection{Application to Diffuse Optical Tomography}
We apply the considered algorithms to the severely ill-posed inverse problem of imaging hemodynamic changes with DOT~\cite{arridge1999optical,Nissila_DOI2005}. DOT is a functional near-infrared spectroscopic imaging method, where the relative transparency of tissue and differing absorption spectra of oxygenated (HbO$_2$) and deoxygenated (HbR) hemoglobin to near-infrared light are utilized to image changes in HbO$_2$, HbR and total hemoglobin (HbT) concentrations. We consider a frequency-domain (FD) instrument, which measures the amplitude and phase shift of the detected photon density waves~\cite{nissila2006comparison}, with a high-density (HD) optode layout. We focus on DOT as a functional brain imaging method, where changes in HbT reflect the vascular and metabolic responses to neuronal activity in the cortex~\cite{Nissila_DOI2005,maria2022imaging}. DOT provides a promising, convenient alternative for imaging small children and has been used, e.g., to study responses to affective touch, emotional speech and audio-visual stimuli in healthy neonates~\cite{jonsson2018,shekhar2019,maria2020,collins2024whole}, and to monitor cerebral hemorrhage in neonates at the intensive care unit~\cite{singh2014mapping}; for a review, see~\cite{lee2017diffuse}. 

DOT image reconstruction is typically based on linearizing the relationship between the measured changes in the detected amplitude and phase shift, and the corresponding hemodynamic changes. In this work, the required sensitivity profiles are obtained with the Monte Carlo eXtreme (MCX) software by simulating the propagation of ten billion photon packets in the head~\cite{Fang2009,yao2018replay,yan2020,Hirvi_2023}. The simulations require a model of the head anatomy, tissue segmentation and optical parameters, along with the source and detector locations and properties, all of which can be modeled with limited accuracy. We focus on the challenges in modeling the head anatomy.  

MRI is the most typical method for obtaining anatomical images of healthy neonates. Yet, in DOT, atlas-based head models have been suggested as an alternative to the individual's own anatomical model~\cite{heiskala2009probabilistic,Custo2010,cooper2012validating,Hirvi_2023}. This requires that we know the head shape of the neonate with some accuracy and can register the atlas model to the target geometry. If the exact exterior surface of the target is known, the registration can be exact, which may require a semi-manual approach that is laborious when considering a large study population. If the exact surface is not known or not forced in the registration, the registered atlas model does not have the exactly correct head shape. Often, the head surface is represented with a set of fiducial points measured with photogrammetry or a digitizing pen~\cite{maria2022imaging,collins2024whole,autti2025simultaneously}.

Avoiding the need to image and segment the individual's own anatomical MR images enables the relatively practical measurement arrangements and simplifies the data analysis pipeline in DOT. Furthermore, relatively comfortable measurements and the fact that the subject does not have to be alone are valuable especially when imaging neonates and toddlers. A limitation of atlas-based models is that they provide only an approximation of the true anatomy of the neonate~\cite{cooper2012validating}. We recently studied the effects of this inner anatomy variation on the FD measurements and difference data using the database of 215 head models in~\cite{CollinsJones2021}, limiting to late preterm to post-term neonates~\cite{Hirvi_2023}. In this work, we consider both the differences in the inner anatomy and the inaccuracy in the exterior boundary related to registration. 

The atlas-based model typically represents either another individual of the same age or an age-appropriate, population-level average. Our research question is, whether we can utilize a database of anatomical models more comprehensively. Previously, Heiskala et al.~\cite{heiskala2009probabilistic} introduced the idea of a probabilistic atlas for modeling the anatomical variation in an atlas of neonatal head models. To be more specific, they first registered seven anatomical models exactly to the target exterior surface using a semi-manual, non-linear method. During the MC simulations, one of the models was selected randomly for each photon packet when computing the Jacobians. Shi et al.~\cite{shi2011infant} generated tissue probability maps for neonates and one- and two-year-olds by registering 95 anatomical models simultaneously to a common space, which did not correspond to the exact head shape of any individual. The registration algorithm utilized inner anatomical structures~\cite{wu2012feature}. For computing the Jacobians, the segmentation of each voxel can be fixed or continuously varying optical parameters can be used to convey more information about the underlying variation.

In this work, we use the 215 segmented neonatal voxel models from~\cite{CollinsJones2021} and consider a relatively simple, automatic linear approach to register the atlas models to the target neonate using only landmarks and fiducial points from the head surfaces. Linear registration (here a combination of rotation, translation and scaling) preserves relative tissue thicknesses and volume ratios, but leaves some variation in the registered head geometries. Thus, there is some ambiguity in the registered heads' exterior and interior anatomy in comparison to the target. We then compute the Jacobians separately for each model to form an atlas of Jacobians. We take the average Jacobian as the initial guess and attempt to improve it by moving along PC directions in the Jacobian space limiting the field-of-view (FOV) with the exterior boundary of the target. This approach is justified if the true target Jacobian can be presented as a linear combination of the mean Jacobian and the PCs with reasonable accuracy.

\subsection{Outline}
This text is organized as follows. Our bilinear inverse problem is introduced in Section~\ref{sec:bilinear_inverse_problem}, including its Bayesian formulation under independent Gaussian priors. Section~\ref{sec:MAP} considers computing a MAP estimate for the joint posterior of the actual unknowns and the expansion coefficients and explains our adaptations of the two optimization algorithms.
Section~\ref{sec:CM} introduces a two-block Gibbs sampling approach to computing a {\em conditional mean} (CM) estimate for the joint posterior and Section~\ref{sec:BAE} interprets the BAE method for our bilinear model. DOT is introduced in Section~\ref{sec:DOT}, and our algorithms are applied to its linear inverse problem under a PCA model for the subject's head anatomy in Section~\ref{sec:numerics}. The conclusions are presented in Section~\ref{sec:conclusion}. Finally, an appendix presents the marginalization of the joint posterior distribution with respect to the expansion coefficients in the forward operator.

\section{Bilinear Finite-Dimensional Inverse Problem}
\label{sec:bilinear_inverse_problem}
This section introduces the deterministic and Bayesian formulations for our bilinear inverse problem. We mainly follow the notational convention for finite-dimensional bilinear operators in~\cite{Elden2018}.

\subsection{Modeling Forward Operator Inaccuracy}
Consider a finite-dimensional linear inverse problem
\begin{equation}
\label{eq:inv_prob}
    A x = b\, ,
\end{equation}
where $x \in \R^{n}$ is the unknown, $b \in \R^{l}$ carries the data, and $A \in \R^{l \times n}$ is the forward operator that is presumably ill-conditioned and inaccurately known. Let $A_0$ denote the mean over a sample of candidate forward operators $\{A_{i} \}_{i=1}^m \subset \R^{l \times n}$. The straightforward approach to tackling the uncertainty in $A$ would be to substitute $A_0$ for $A$ in \eqref{eq:inv_prob} and aim to solve the resulting approximate inverse problem. 

Under the assumption that the set $\{A_{i} \}_{i=1}^m$ provides a comprehensive description of the possible variations in $A$, employing the sample mean ignores a lot of available information. In this work, we consider representing $A$ as a linear combination of given basis matrices; as a motivation for such an expansion, consider PCA on the sample operators, with the aim to include more prior information on the forward operator in the solution process. In consequence, we present the unknown forward operator as  
\begin{equation}
\label{eq:PCA}
    A \approx A_{0} + \sum_{i=1}^p y_{i} V_{i}\, ,
\end{equation}
where $V_{i} \in \R^{l \times n}$ are the considered basis matrices and $y_{i} \in \R$ are the unknown coefficients that we aim to estimate as a part of an {\em extended} inverse problem. The price one has to pay for such an extension is that \eqref{eq:inv_prob} becomes bilinear in $(y,x) \in \R^{p} \times \R^{n}$.\footnote{We slightly abuse the terminology by calling the extended inverse problem bilinear although it is linear in $x$ and only affine in $y$.}

\begin{remark}
\label{remark:PCA}
There are several possible ways to end up with a representation of the form \eqref{eq:PCA}. In the context of PCA, the most straightforward approach, adopted for our implementation of the BAE method, is to consider $\{A_{i} \}_{i=1}^m$ in the Frobenius topology, which reduces the computations to standard PCA in $\R^{ln}$. For the other algorithms, we instead perform row-wise PCA for the sample matrices, leading to principal component (PC) matrices $V_i$ each of which only has a single nonempty row. This choice is later motivated by physical intuition on the measurement setup of DOT as well as a relatively small sample size $m$. We emphasize that the coefficients $y_i$ are treated as regularized nuisance variables rather than as identifiable physical parameters: their Gaussian prior variances are inherited from the PCA, and the purpose of including them is to allow adaptation along typical Jacobian-variation directions, not to reconstruct a physiological subject-specific Jacobian.
The novelty in the present work lies in applying these PCA choices to a DOT Jacobian atlas constructed from registered neonatal head models, not in the bilinear parameterization itself.
\end{remark}

\subsection{Notation and Deterministic Formulation}
The proposed extended formulation for the considered inverse problem can be given in a tensor form, cf.~\cite{Elden2018}. Define $\Acal \in \R^{l\times p \times n}$ by demanding that $\Acal (\, \cdot \,,i, \, \cdot \,) = V_{i}$ for $i = 1, \dots, p$. That is, fixing the middle index in $\Acal$ reproduces the corresponding basis matrix in~\eqref{eq:PCA}. It follows that \eqref{eq:PCA} can be written equivalently as
\begin{equation}
\label{eq:PCA_tensor}
    A \approx A_0 + \Acal \bigcdot (y)_2 = A_0 + A_{y,2} \, ,
\end{equation}
where $y \in \R^{p}$ and $A_{y,2} = \Acal \bigcdot (y)_2$. Analogously, the considered inverse problem \eqref{eq:inv_prob} can be expressed as
\begin{equation}
\label{eq:tensor_inv_prob}
    A_0 x + \Acal \bigcdot (y,x)_{2,3} = b.
\end{equation}
We also use the notation
\begin{equation*}
    A_{x,3} = \Acal \bigcdot (x)_{3} =
    \big[ V_{1} x  \ \  V_{2} x  \ \cdots \ V_{p} x
    \big] \in \R^{l \times p}.
\end{equation*}
In what follows, $\| \, \cdot \, \|$ denotes the Euclidean norm of the considered real vector space. 

Requiring $A_0 \not= 0$ is essential for avoiding a generic nonuniqueness~\cite{Elden2018} for the unregularized problem~\eqref{eq:tensor_inv_prob}: If $A_0 = 0$ and $(y,x)$ is a solution of \eqref{eq:tensor_inv_prob}, then also $(\alpha^{-1} y,\alpha x)$ is a solution for any $\alpha \not= 0$; see also \cite{Johnson09,Johnson14}.

\subsection{Bayesian Formulation} 
We are interested in tackling the inverse problem of deducing information on $x$ from \eqref{eq:tensor_inv_prob} in the Bayesian framework. Hence, we consider a probabilistic extension of \eqref{eq:tensor_inv_prob},
\begin{equation}
\label{eq:Btensor_inv_prob}
    B = A_0 X +  \Acal \bigcdot (Y,X)_{2,3} + E\, ,
\end{equation}
where the random variables $B$, $Y$ and $X$ correspond to $b$, $y$ and $x$, respectively, and $E \in \R^l$ models additive measurement noise. In particular, the uncertainty, or randomness, in the forward operator is represented by $Y$. We assume that  $E$, $Y$ and $X$ are mutually independent Gaussian random variables,
\begin{equation}
\label{eq:priors}
        E \sim \mathcal{N}(0,\,\Gamma_1), \quad
        Y \sim \mathcal{N}(0,\,\Gamma_2),
        \quad
        X \sim \mathcal{N}(0,\,\Gamma_3), 
\end{equation}
where $\Gamma_1 \in \R^{l \times l}$, $\Gamma_2 \in \R^{p \times p}$ and $\Gamma_3 \in \R^{n \times n}$ are the respective symmetric positive definite covariance matrices. Assuming that the additive noise is independent of the unknowns can be considered standard in Bayesian inverse problems, whereas the reasonability of the assumed independence of $Y$ and $X$ depends on the application. We will motivate the latter assumption in the framework of DOT in Section~\ref{sec:abs_prior}. 

The joint posterior probability density for the two unknowns can be written with the help of Bayes' formula as
\begin{align}
\label{eq:xy_posterior}
    \pi (y,x \mid b) &\propto \pi (y,x) \pi(b \mid y,x) = \pi (y) \pi (x) \pi \big (b - A_0x - \Acal \bigcdot (y,x)_{2,3} \big) \nonumber \\[1mm]
    &\propto \exp{\left( - \dfrac{1}{2} \Big( \lVert b - A_0x - \Acal \bigcdot (y,x)_{2,3}  \rVert^{2}_{\Gamma_1^{-1}} +\lVert y \rVert^{2}_{\Gamma_2^{-1}} + \lVert x \rVert^{2}_{\Gamma_3^{-1}}   \Big) \right)},
\end{align}
where the dropped constants are independent of $y$ and $x$ and the standard notation $\| z \|_{M}^2 = z^\top \! M z$ is used for a symmetric positive definite matrix $M$.

\begin{remark}
Using independent Gaussian priors for \eqref{eq:Btensor_inv_prob} yields quadratic `penalty terms' in the exponent of the posterior density \eqref{eq:xy_posterior}. One could also use other kinds of prior information on $y$ and $x$; e.g.,~if $y$ and $x$ were, respectively, assumed to be sparse and composed of well-located perturbations in constant background, then a natural choice would be to employ $\ell_1$ and total variation priors that have in many studies shown to be superior to Gaussian priors under such assumptions. However, as this work presumably presents the first study on using a bilinear forward model of the type \eqref{eq:Btensor_inv_prob} to mitigate physiological model uncertainties in an imaging problem, we consider adopting a simple prior model a natural choice.
\end{remark}

In this work, we concentrate on algorithms for approximating joint MAP or CM estimates for $X$ and $Y$ based on the posterior density \eqref{eq:xy_posterior}. However, for completeness we point out that under the above assumptions, the posterior density of $X$ can, in fact, be given explicitly via marginalization over $Y$.

\begin{proposition}
\label{prop:marg_post}
Assume the measurement model \eqref{eq:Btensor_inv_prob} and that $E$, $Y$ and $X$ are {\em a priori} mutually independent with the distributions \eqref{eq:priors}. Then, the posterior density of $X$ given $B = b$ is
\begin{align}
\label{eq:marg_post}
    \pi(x \mid b) \propto \big| \Gamma_{y | b, x} \big|^{1/2}
    \exp \Big( \! -\frac{1}{2} \big( \| x \|^2_{\Gamma_3^{-1}}  
   + \| b - A_0 x \|^2_{\Gamma_{b|x}^{-1}}  \big) \Big),
\end{align}
where the omitted constant does not depend on $x$ and
\begin{equation*}
\Gamma_{y | b, x} = \big(\Gamma_2^{-1} + A_{x,3}^{\top}\Gamma_1^{-1}A_{x,3}\big)^{-1} \quad \text{and} \quad \ \Gamma_{b|x} = \Gamma_1 + A_{x,3} \Gamma_2 A_{x,3}^{\top} 
\end{equation*}
are, respectively, the covariance matrices of the Gaussian distributions for $Y$ given $B=b$ and $X = x$ and for $B$ given  $X = x$.
\end{proposition}

\begin{proof}
   The proof is presented in Appendix~\ref{sec:marg_post}. 
\end{proof}

We complete this section with a brief remark on how reconstructions of $x$ are computed if the forward operator is fixed.

\begin{remark}
\label{remark:fixed}
When a fixed forward operator (or Jacobian matrix) is considered in the numerical examples of Section~\ref{sec:numerics}, the inverse problem is solved for $x$ assuming no variation in $y$. In such a case, the reconstruction is the MAP/CM estimate for the probability density
\[
\pi (x \mid b)\propto \exp{\left( - \dfrac{1}{2} \Big( \lVert b - \hat{A} x \rVert^{2}_{\Gamma_1^{-1}} + \lVert x \rVert^{2}_{\Gamma_3^{-1}}   \Big) \right)},
\]
where the matrix $\hat{A}$ should be clear from the context. Moreover, when applying the BAE method to \eqref{eq:Btensor_inv_prob}, $\Gamma_1$ is replaced by $\Gamma_1 + \Gamma_{\varepsilon}$; see Section~\ref{sec:BAE} for the details.
\end{remark}

\section{On computing a MAP estimate}
\label{sec:MAP}
Maximizing the posterior density $\pi (x,y \mid b)$ is equivalent to minimizing the (scaled) negative log-posterior
\begin{equation}
\label{eq:Tikhonov}
    \Phi(y,x) = \big \lVert  b - A_0x - \Acal \bigcdot (y,x)_{2,3} \big \rVert^{2}_{\Gamma_1^{-1}} + \lVert y \rVert^{2}_{\Gamma_2^{-1}} + \lVert x \rVert^{2}_{\Gamma_3^{-1}}.
\end{equation}
In this section, we briefly consider properties of the Tikhonov functional \eqref{eq:Tikhonov} and adapt the BCD and G--N algorithms for minimizing it.

The objective $\Phi$ in \eqref{eq:Tikhonov} is biconvex but not jointly convex in general; even the scalar case can possess multiple local minima, as illustrated by the following example.

\begin{example}
\label{example:2D}
Let $n=p=1$, $b=1$, $A_0 = 1$, $\Gamma_1 = 1$, $\Gamma_3^{-1} = \Gamma_2^{-1} = \beta \in \R_+$ and define $\mathcal{A}: \R \times \R \ni (y,x) \mapsto xy \in \R$. These choices result in
\begin{equation}
\label{eq:simple_Tikhonov}
\Phi(y,x) = (1 - x - yx)^2 + \beta ( y^2 + x^2).
\end{equation}
For $\beta=0.1$, $\Phi$ has two local minima and one saddle point, whereas for $\beta=1$ it has only one critical point, which is the global minimum; see Fig.~\ref{fig:example1}. Hence, stronger regularization or higher noise level reduces the effect of nonconvexity in this simple example.
\end{example}

\begin{figure}[t!]
    \centering
 \includegraphics[width=\textwidth]{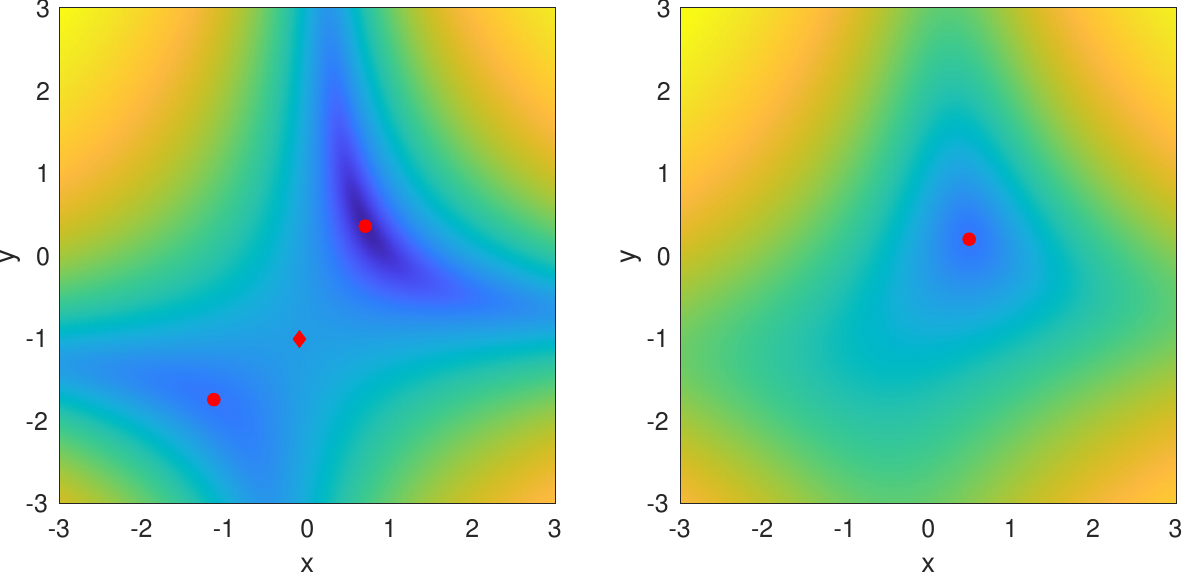}
    \caption{Visualization of $\Phi(y,x)$ from \eqref{eq:simple_Tikhonov}. Left: $\beta = 0.1$, the diamond corresponds to a saddle point and the dots mark the local minimizers. Right: $\beta = 1$ and the dot marks the global minimizer.}
    \label{fig:example1}
\end{figure}

\subsection{Block Coordinate Descent}
\label{sec:alternating}

Our first iterative algorithm for minimizing $\Phi(y,x)$ in \eqref{eq:Tikhonov} is BCD. Given the current iterate $(y_k,x_k)$, we fix one block and minimize with respect to the other:
\begin{equation}
\label{eq:x_min}
     x_{k+1} =
     \underset{x \in \R^n}{\argmin} \;
     \big \lVert b - (A_0 + A_{y_k,2} ) x \big \rVert^{2}_{\Gamma_1^{-1}} + \lVert x \rVert^{2}_{\Gamma_3^{-1}},
\end{equation}
and
\begin{equation}
\label{eq:y_min}
     y_{k+1} =
     \underset{y \in \R^p}{\argmin} \;
     \big \lVert ( b - A_0 x_{k+1} )  - A_{x_{k+1},3}  y \big \rVert^{2}_{\Gamma_1^{-1}} + \lVert y \rVert^{2}_{\Gamma_2^{-1}}.
\end{equation}
Both subproblems are standard quadratic Tikhonov problems corresponding to linear Gaussian models. Algorithm~\ref{alg:alternating} summarizes the resulting minimization scheme; for related convergence analysis, see,~e.g.,~\cite{Xu13} and the references therein. The updates are written in data-space form using the Woodbury matrix identity, which is efficient when the number of measurements is low compared to the dimensions of $x$ and $y$.

\begin{own_algorithm}
\label{alg:alternating}
\begin{algorithmic}
  \STATE{Select the initial guess $y_0$ and stopping criterion, and set $k=0$.}

 \vspace{1mm}

  \WHILE{Stopping criterion is not met}
\STATE{
\vspace{-5mm}
\begin{align*}
&x_{k+1} = \Gamma_3 (A_0 + A_{y_k,2})^\top \big( (A_0 + A_{y_k,2}) \Gamma_3 (A_0 + A_{y_k,2})^\top + \Gamma_1\big)^{-1} b, \\[0mm]
&y_{k+1} = \Gamma_2 A_{x_{k+1},3}^\top \big( A_{x_{k+1},3} \Gamma_2 A_{x_{k+1},3}^\top + \Gamma_1\big)^{-1} (b - A_0 x_{k+1}), \\[1mm]
&k  \leftarrow k+1.
\end{align*}
\vspace{-5mm}
}
  \ENDWHILE

  \vspace{1mm}

  \RETURN $(y_k,x_k)$.
\end{algorithmic}
\end{own_algorithm}

\subsection{Gauss--Newton Algorithm}
\label{sec:GN}

The second algorithm for minimizing $\Phi(y,x)$ in \eqref{eq:Tikhonov} is the G--N method. We linearize the map $
(y,x)\mapsto A_0x+\Acal \bigcdot (y,x)_{2,3}$
around the current iterate $(y_k,x_k)$. The corresponding Jacobian is
\begin{equation}
\label{eq:Jacobian}
    J_{\mathcal{A}}(y_k,x_k) = 
\big[A_{x_k,3} \ \ A_0 + A_{y_k,2} \big],
\end{equation}
and the resulting regularized linear least-squares problem can be written as
\begin{equation}
\label{eq:GN_LS2}
 \left\lVert  \big( b + \Acal \bigcdot (y_k,x_k)_{2,3} \big)  -  J_\mathcal{A}(y_k,x_k) \begin{bmatrix} y  \\ x \end{bmatrix}  \right \rVert^{2}_{\Gamma_1^{-1}} + 
\big[y^\top  \ x^\top \big] \Gamma_{2,3}^{-1}  \begin{bmatrix} y  \\ x \end{bmatrix},
\end{equation}
where $\Gamma_{2,3}^{-1} = {\rm diag}(\Gamma_2^{-1}, \Gamma_3^{-1})$. Interpreting \eqref{eq:GN_LS2} as a linear-Gaussian inverse problem gives the preliminary update rule
\begin{equation}
\label{eq:GN_inversion}
\begin{bmatrix} y_{*}  \\ x_{*} \end{bmatrix}
= \Gamma_{2,3} J_\mathcal{A}(y_k,x_k)^{\top} \big( J_\mathcal{A}(y_k,x_k) \Gamma_{2,3} J_\mathcal{A}(y_k,x_k)^\top + \Gamma_1 \big)^{-1} \big( b + \Acal \bigcdot (y_k,x_k)_{2,3} \big),
\end{equation}
which only involves operating with the inverse of a matrix with dimension $l$. The step
\begin{equation}
\label{eq:GN_update}
p_{k+1} = \begin{bmatrix} y_{*}  \\ x_{*} \end{bmatrix} - \begin{bmatrix} y_{k}  \\ x_{k} \end{bmatrix}
\end{equation}
is damped by $\kappa \in (0,1]$ that could in principle be chosen adaptively, cf.~\cite{Elden2018}.
The complete G--N scheme for minimizing $\Phi(y,x)$ is summarized in Algorithm~\ref{alg:GN}.

\begin{own_algorithm}
\label{alg:GN}
\begin{algorithmic}
  \STATE{Select the initial guess $(y_0, x_0)$, the stopping criterion, and a step size parameter $\kappa \in (0, 1]$. Set $k=0$.}

 \vspace{1mm}

  \WHILE{Stopping criterion is not met}
\STATE{\hspace{2mm} $\rhd$ Form $J_\mathcal{A}(y_k,x_k)$ via \eqref{eq:Jacobian}.}
\STATE{\hspace{2mm} $\rhd$ Compute $p_{k+1}$ via \eqref{eq:GN_update}.}
\STATE{\hspace{2mm} $\rhd$ Set $$
\begin{bmatrix} y_{k+1}  \\ x_{k+1} \end{bmatrix} = \begin{bmatrix} y_{k}  \\ x_{k} \end{bmatrix} + \kappa p_{k+1}.
$$
}
  \ENDWHILE

  \vspace{1mm}

  \RETURN $(y_k,x_k)$.
\end{algorithmic}
\end{own_algorithm}

\section{Conditional Mean Estimate}
\label{sec:CM}
An alternative approach to solving the Bayes\- ian inverse problem \eqref{eq:Btensor_inv_prob} is to estimate the expected value of the posterior \eqref{eq:xy_posterior},~i.e., compute the CM estimate, by some Markov Chain Monte Carlo (MCMC) integration method. 
We only consider a variant of the Gibbs sampler in this work.

\subsection{Gibbs Sampler}
Following,~e.g.,~\cite[Section~3.6.3]{Kaipio2006}, we implement a Gibbs sampler for the joint posterior \eqref{eq:xy_posterior} by partitioning the unknown $(y,x)$ into the two natural blocks $y$ and $x$. Fixing either block turns \eqref{eq:xy_posterior} into a linear-Gaussian inverse problem for the other block. Thus, the conditional distributions $\pi(x\mid b,y_k)$ and $\pi(y\mid b,x_{k+1})$ are Gaussian, with means and covariances given by the standard Gaussian conditioning formulas. In particular, the conditional means have similar form as the two BCD updates in Algorithm~\ref{alg:alternating}, while the corresponding covariances are obtained from \cite[Theorem~3.7]{Kaipio2006}. A two-block Gibbs sampler can then be implemented by making random draws in turns from these two Gaussian distributions.

\begin{own_algorithm}
\label{alg:Gibbs}
\begin{algorithmic}
  \STATE{Select the initial guess $y_0$ and the sample size $K$.}

 \vspace{1mm}

  \FOR{$k=0, \dots, K-1$}
\STATE{\hspace{2mm} $\rhd$ Draw $x_{k+1}$ from the density $\pi(x\mid b,y_k)$.}
\STATE{\hspace{2mm} $\rhd$ Draw $y_{k+1}$ from the density $\pi(y\mid b,x_{k+1})$.}
  \ENDFOR

  \vspace{1mm}

  \RETURN $\{(y_k, x_k)\}_{k=1}^{K}$.
\end{algorithmic}
\end{own_algorithm}

The sample mean and covariance of the sample $\{(y_k, x_k)\}_{k=1}^{K}$ produced by Algorithm~\ref{alg:Gibbs} can be used as estimates for the mean and covariance of the posterior~\eqref{eq:xy_posterior}.

\section{Approximate marginalization by the Bayesian approximation error method}
\label{sec:BAE}

The BAE method \cite{Kaipio2006} has previously been applied to many imaging problems under different types of uncertainties; see \cite{Arridge_2006,Kolehmainen2009,Mozumder2013,Tarvainen2010,Tarvainen2009} for applications to DOT. We use it here as an approximate way to marginalize the $Y$-dependence in the bilinear model \eqref{eq:Btensor_inv_prob}. Namely, we write
\begin{align}
\label{eq:Btensor_inv_prob_AE}
    B = A_0 X  + \mathcal{E} + E,
\end{align}
where 
\[
\mathcal{E} =
\big (A_0 X + \Acal \bigcdot (Y,X)_{2,3} \big)
-
A_0X
=
\Acal \bigcdot (Y,X)_{2,3}
\]
is the approximation error noise. When considering the BAE method in the following sections, $Y$ corresponds to {\em all} coefficients in a Frobenius-norm based PCA for the sample $\{A_{i} \}_{i=1}^m$, cf.~Remark~\ref{remark:PCA}.

We adopt the enhanced error model, treating $\mathcal{E}$ as Gaussian and independent of $X$ and $E$. Although these assumptions are only approximate because of the bilinearity of~$\Acal$, the enhanced error model has proven useful in other settings with modeling errors; see \cite{Candiani2021} for another application to head-imaging. Since $X$ and $Y$ are independent zero-mean Gaussians, the mean of $\mathcal{E}$ is zero. Its covariance is estimated by drawing $N$ samples from the priors of $X$ and $Y$, say, $\{ x_j \}_{j=1}^N$ and $\{ y_j \}_{j=1}^N$, and computing
\begin{equation}
\label{eq:sample_cov}
\Gamma_{\varepsilon} = \frac{1}{N-1} \sum_{j=1}^N \Acal \bigcdot (y_j,x_j)_{2,3}  \big( \Acal \bigcdot (y_j,x_j)_{2,3} \big)^{\top}.
\end{equation}  
As $\mathcal{E}$ and $E$ are assumed to be independent, the MAP/CM estimate for $X$, induced by \eqref{eq:Btensor_inv_prob_AE} and the measurement $B = b$, can be efficiently computed by setting $\hat{A} = A_0$ and replacing $\Gamma_1$ by $\Gamma_1 + \Gamma_{\varepsilon}$ in Remark~\ref{remark:fixed}.

\section{Linearized Diffuse Optical Tomography}
\label{sec:DOT}

We apply the presented algorithms to imaging simulated activity in a neonatal head with DOT using an atlas of candidates for the unknown inner anatomy. In the context of \eqref{eq:inv_prob}, the unknown of primary interest $x$ contains the changes in the absorption coefficients of each voxel, or optically homogeneous 3D pixel that the head is discretized into. Absorption coefficient changes reconstructed with measurements at a single near-infrared wavelength are directly convertible into changes in HbT at the wavelength of 798\,nm, since the extinction coefficients of HbO$_2$ and HbR are the same~\cite{cope1991development}. The data vector $b$ contains the FD difference measurements. Finally, the forward operator $A$ is the sensitivity profile, or Jacobian matrix, where the rows correspond to the measurements and the columns contain the derivatives of the measurement with respect to per-voxel absorption coefficients. The linearized model is motivated by the assumption that the changes in absorption are relatively low in contrast~\cite{Nissila_DOI2005}.     

The sample of candidate forward operators $\{A_{i} \}_{i=1}^m \subset \R^{l \times n}$ is obtained by simulating the Jacobians separately in each of the 215 neonatal head models, originally from the library in~\cite{CollinsJones2021} but registered linearly to the same reference head shape. The registration process is described below. None of the linearly registered models presents the inner anatomy or head shape of another individual exactly, but we hope that the atlas of 214 models enables our algorithms to find a better forward operator for each registered individual than the mean Jacobian over all other neonates.

The FD difference measurements consist of two data types, log-amplitude ($\ln {\sf A}$) and phase shift ($\varphi$), and we can reformulate the previous formulas to explicitly present them by using matrices with block structure. We assume that the noise for the different data types are mutually independent, so that the covariance matrix $\Gamma_1$ can be written as $\Gamma_1 = {\rm diag}(\Gamma_1^{\ln {\sf A}}, \Gamma_1^{\varphi}) \in \R^{2l \times 2l}$. Since the PCA is performed independently for the two data types, we can also write $\Gamma_2 = {\rm diag}(  \Gamma_2^{\ln {\sf A}}, \Gamma_2^{\varphi} ) \in \R^{2p \times 2p}$. Moreover, by letting $\Acal \in \R^{2l \times 2p \times n}$, and fixing either $y$ or $x$, we obtain the matrices
\begin{equation}
    A_{y,2} = 
    \begin{bmatrix}
        A_{y,2}^{\ln {\sf A}} \\[1mm]
        A_{y,2}^{\varphi}
    \end{bmatrix} \in \R^{2l \times n}, \quad A_{x,3}  = {\rm diag} \big( A_{x,3}^{\ln A},  A_{x,3}^{\varphi}\big)
     \in \R^{2l \times 2p}.
\end{equation}
The expected forward operator, and the measurements are expressed as
\begin{equation}
    A_0 = 
    \begin{bmatrix}
        A_{0}^{\ln {\sf A}} \\[1mm]
        A_{0}^{\varphi}
    \end{bmatrix} \in \R^{2l \times n},
    \quad \mathrm{and} \quad
    b = 
    \begin{bmatrix}
        b^{\ln {\sf A}} \\ 
        b^{\varphi}
    \end{bmatrix} \in \R^{2l}.
\end{equation}
With these notations, the posterior density and reconstruction algorithms are as presented in previous sections.

As pointed out in Remark~\ref{remark:PCA}, we perform two different PCAs for the log-amplitude and phase shift Jacobians. The one used for the optimization algorithms and Gibbs sampling is carried out on rows of these Jacobians,~i.e.,~on the gradients of individual measurements with respect to the voxel-wise absorption coefficients. The Jacobian for one measurement, i.e., one source--detector pair, typically has highest sensitivity along a ``banana-shaped'' region between the optodes, with the depth of the curve depending on the source--detector separation (SDS) as well as the measurement type. The anatomy --- for example the location of sulci --- affects the profile. Our motivation for performing the PCA row-wise is computational efficiency and that the variation in these profiles is a more intuitive target for the PCA compared to the whole Jacobian $A$, for which the shortest channels with the highest sensitivities could dominate. According to our experience, this approach also functions better for the BCD and G--N algorithms. On the other hand, the PCA used in connection to the BAE is performed in the Frobenius topology because considering the individual measurements separately when propagating the uncertainty in $y$ through the bilinear model would hide correlations in the approximation error noise $\varepsilon$ between different measurement channels. This is a computationally heavier preprocessing step than the row-wise PCA. The PCAs for the log-amplitude and phase are still computed separately due to the different magnitudes and units of the corresponding measurements.

\subsection{Models}
We used the neonatal head model database presented in~\cite{CollinsJones2021} that was created in collaboration between University College London (UCL) and the Centre for the Developing Brain at King's College London, using data from the Developing Human Connectome Project (dHCP). As initial preparations, we present the segmented voxel models, the landmarks and other cranial reference points in 0.5-mm resolution in the head coordinate frame defined by the landmarks as described in~\cite{CollinsJones2021}. The five manually-selected landmarks are the nasion (Nz), inion (Iz), right (RPA) and left pre-auricular points (LPA), and the approximate highest point on the top of the head (Cz). The 249 automatically-defined cranial reference points are a subset of the 345 positions according to the 10--5 system, which is a standard method for placing electrodes in HD electroencephalography (EEG)~\cite{oostenveld2001five,CollinsJones2021}. The voxel models for all neonates are fixed into common matrix size, with the center of matrix at the approximate head center as defined by Iz and Cz. We refine the segmentation of the cerebrospinal fluid (CSF) by separating the subarachnoid layer from the CSF in the sulci and ventricles, and narrow the number of brain tissue types to only include gray (GM) and white (WM) matter, as described in~\cite[Sections~3.1--3.2]{Hirvi_2023}. In this work, we include all of the 215 early preterm to post-term neonates with combined gestational and chronological age range of 29.3--44.3 weeks to maximize the sample size.  

We assume to know the head shape of the target neonate as given by the landmarks and other reference points detected from the target head surface. To form the sample of candidate head models and eventually Jacobians, we use the reference points to register the other neonates' head models to the target. For simplicity, instead of forming 215 sets of head models registered to a different geometry, we selected one of the neonates as the reference head shape. Now, it is sufficient to register each of the remaining 214 neonates only once to a target head shape and simulate the Jacobians once. We used the same median-sized, 41.7-week (combined gestational and chronological age) reference neonate as motivated in~\cite[Section~3.3]{Hirvi_2023}. In the following analysis, we can still consider any of the registered newborns as the target to approximate the performance of our algorithms for varying models, although the optimal approach would be to perform the registrations separately for each target.

\subsubsection{Head Model Registration}
\label{sec:registration}
We register each of the 214 neonates to the reference by matching the corresponding landmarks and the 10--5 reference points. As the initial registration, we match the respective Iz--Nz midpoints, place the origin at the reference Nz, and solve a combined rotation and translation using only the landmarks. The goal is to match the orientation of the head to the reference, since the poses may not initially match due to variations in the selected locations of the anatomical landmarks. A typical initial orientation mismatch seems to be a slight tilt due to variation in the selected LPA and/or RPA locations. 

The optimal rotation and translation are solved in the least-squares sense with the singular value decomposition as explained in~\cite{arun1987}. The implemented algorithm is also known as 
the Kabsch/Kabsch-Umeyama Algorithm~\cite{kabsch1976solution,umeyama1991least,lawrence2019purely}, where it usually refers to solving the optimal rotation matrix after centering the point sets, i.e., moving the respective centroids to the origin. However, here we solve both the rotation (for centered points) and translation explicitly. 

We compare solving the transformation with all landmarks, versus excluding one landmark at a time. If the best initial registration is obtained without one of the landmarks, the respective landmark is replaced by projecting the reference landmark on the neonate's head surface after the transformation. Thus, here we assume that (only) one landmark position can be less accurate and attempt to fix it. The 10--5 reference points have been computed automatically based on the landmarks~\cite{CollinsJones2021}, thus it would be challenging to correct the orientation without modifying the reference points if all point data were used instead of the landmarks.

After fixing the head orientation, we solve a combined scaling and translation that minimizes the sum of the surface ($\mathrm{SRE}$) and the landmark registration error ($\mathrm{LRE}$) (cf.~\cite[Eq.~(5)]{koikkalainen2004} and \cite[Eq.~(20)]{Hirvi_2023}), defined as the mean squared Euclidean distances between the 10--5 reference points and the landmarks, respectively. The transformation is solved in the least-squares with the same idea as in~\cite[Eq.~(21)]{Hirvi_2023}, where the points are weighted to implement the averaging.

The described registration approach was selected since solving a full affine transformation between all fiducial points can lead  to an unwanted shear transformation. 
The focus was to place the measurements over the same region as guided by the anatomical landmarks. After the registration was completed in 0.5-mm resolution, the models were converted to 1-mm voxel side length for the following steps. 

\subsubsection{Optodes and Surface Phenomena}
\label{sec:optodes}
We modified the optode layout of an actual HD measurement probe used with the DOT instrument at Aalto University (see~\cite[Fig.\ 2.b]{maria2022imaging}) by adding two detectors on each row of detectors to further increase the density of optode pairs and include shorter channels. The modified probe includes 15 sources and 21 detectors. 

The probe was projected to the left hemisphere of the reference head and the optode locations were optimized along the surface to restore the original optode separations using the algorithm from~\cite{hirvi2019master}, also applied in~\cite{maria2022imaging, SHEKHAR2024177, autti2025simultaneously}. The Iso2Mesh software (\url{https://github.com/fangq/iso2mesh}; accessed 11/10/2025) was used to create a smooth triangular surface mesh for the optimization~\cite{fang2009tetra,tran2020}. The optodes were further projected on the surface of each registered model individually using MATLAB's isosurface-function to create dense surface meshes that roughly follow the voxel boundaries by splitting voxels~\cite{yan2020}. 

The shortest SDS is approximately 5\,mm, and we only include optode pairs with SDS under 40\,mm, as in~\cite{fogarty2025functional}, to exclude channels with high levels of stochastic noise and manage the memory requirements in the PCA. This gives in total $l=210$ source--detector pairs. We used a collimated Gaussian beam source with waist radius of 1.25\,mm. The source intensity modulation frequency was set as $f=100$\,MHz. Specular reflection was ignored, thus all photon packets enter the tissue with the maximum initial weight of one. The detector area was defined as the intersection of a sphere with radius of 1.82\,mm and the exterior surface, as in~\cite{maria2022imaging,SHEKHAR2024177}. We used the MCX default isotropic detector model and enabled reflections when light exits the tissue. 

\subsubsection{Optical Parameters}
The selected tissue-specific optical parameters are listed in Table \ref{table:param}. We considered values reported for 800\,nm in published literature as close enough to be representative of 798\,nm for the purposes of this paper. As mentioned earlier, at 798\,nm the absorption changes can be scaled to get HbT changes~\cite{cope1991development}. The refractive index was set to a constant 1.4 in all tissue types~\cite{bolin1989refractive,simpson1998near,van1993optical,Fukui2003} to avoid challenges with reflections at irregular inner boundaries. 

\begin{table}[t!]
\centering
\caption{Selected literature-based optical parameters for 798\,nm. CSF1: semidiffusive CSF in the subarachnoid region. CSF2: less diffusive CSF in the sulci and ventricles. Gray and white matter parameters are given separately for the youngest preterm (less than 34 weeks combined gestational and chronological age) and older neonates.}  
\begin{tabular}{ccccc}
 \toprule
 Tissue Type & $\mu_{a}$ [mm$^{\mathrm -1}$] & $\mu_{s}$ [mm$^{\mathrm -1}$] & $g$ & $n$\\ 
 \midrule
  Scalp \& Skull & 0.016 & 16 & 0.9 & 1.4 \\
  CSF1 & 0.004 & 1.6 & 0.9 & 1.4 \\
  CSF2 & 0.002 & 0.01 & 0.9 & 1.4 \\
  Gray Matter & 0.048 & 3/25 & 0.8/0.98 & 1.4 \\
  White Matter & 0.037 & 3/50 & 0.8/0.98 & 1.4 \\	 
 \bottomrule	
\end{tabular}
\label{table:param}
\end{table}

The parameters for the combined scalp \& skull layer are roughly the averages of the values reported for adult Caucasian dermis ($\mu_a$ $\approx$ 0.013\,mm$^{-1}$, $\mu_s'$ $\approx$ 1.9\,mm$^{-1}$, $g$ = 0.9) and subdermis ($\mu_a$ $\approx$ 0.009\,mm$^{-1}$, $\mu_s'$ $\approx$ 1.2\,mm$^{-1}$, $g$ = 0.9)~\cite{simpson1998near} and pig skull ($\mu_a$ $\approx$ 0.024\,mm$^{-1}$, $\mu_s'$ $\approx$ 1.9\,mm$^{-1}$, $g$ = 0.94)~\cite{firbank1993measurement} at 800\,nm. These values were originally selected for the skull in~\cite{okada2003,Fukui2003}, and (nearly) the same values were used for the combined scalp \& skull layer in~\cite{jonsson2018}. 

The CSF parameters are as in~\cite{Hirvi_2023}, except that the scattering for the sulci was lowered to the value used in~\cite{okada2003,maria2022imaging} for CSF with less arachnoid trabeculae to slightly highlight the effect of anatomical variation in the location of the sulci. The $g$-factor for both CSF types was set to 0.9, which is the value used in~\cite{okada2003} for the arachnoid trabeculae that are assumed to be the main scatterers in CSF~\cite{okada2003}.

The optical parameters for GM and WM are in accordance with postmortem measurements from the cerebrum of two neonates with gestational ages of 28 and 40 weeks~\cite{van1993optical}. Since ``very little differentiation''~\cite{van1993optical} was observed between GM and WM at 28 weeks, we decided to use partly different optical parameters for the preterm neonates with combined gestational and chronological age under 34 weeks (half-way between at 28 and 40 weeks; 22/215 neonates in total) compared to older gestational ages. In particular, the $g$-factor and reduced scattering coefficient ($\mu_s'$) for combined GM and WM in the early preterm were picked as sample averages at 28 weeks, while separate single-sample values for GM and WM at 40 weeks were used for the older neonates. We utilized the formula $\mu_s' = (1-g)\ \mu_s$ to recover corresponding scattering coefficients. Same GM and WM absorption coefficients were selected for all gestational ages from \cite{Fukui2003}, originally from \cite{van1993optical}.

\subsection{Jacobian Simulation and Principal Component Analysis}
We used the precompiled, latest release version 2025.10 Kilo-Kelvin of the MCX/MCXLAB software for the Monte Carlo simulations. We simulated 10 billion photon packets for each source (15) and each neonate (215). The total detected weight ($I$), amplitude (${\sf A}$) and phase ($\varphi$) can be computed from the final weights of detected photon packets ($w_{p}$) for each source--detector pair as (see~\cite[Eqs.\ (1--3)]{Hirvi_2023})
\[
I =  \frac{1}{N_{\rm input}} \sum_{p} w_{p}, \quad
X = \frac{1}{N_{\rm input}} \sum_{p} w_{p} \cos{(2\pi f t_{p})}, \quad  Y = \frac{1}{N_{\rm input}} \sum_{p} w_{p} \sin{(2\pi f t_{p})},
\]
and
\begin{equation}
\label{eq:MCXforwdata}
{\sf A} = \sqrt{X^2 + Y^2}, \qquad 
\varphi = \arctan(Y,X), 
\end{equation}
where $t_p$ is the time-of-flight for photon packet $p$ and $N_{\rm input}$ is the number of input photon packets, which here effectively normalizes the input source strength to one. The total detected weight $I$ can thus be considered as the overall detection probability for the respective measurement setup, and the weight could be approximately scaled to match real measurements by considering the real input power, fiber losses, etc. 

The Jacobians for the log-amplitude and phase measurements with respect to the voxel-wise absorption coefficients were computed using the ``rf replay'' mode of MCXLAB; see \cite[Section 2.2.1 \& Eqs.\ (8--9)]{Hirvi_2023}. The Jacobians were initially simulated with 1-mm voxel side length, but they were also converted to 2-mm resolution: A column in a Jacobian associated to a 2-mm voxel in the coarser discretization is obtained by summing the columns associated to the 1-mm subvoxels in the corresponding high-resolution Jacobian. The field-of-view (FOV), i.e., the region that is included in the reconstructions, contains voxels for which the sensitivity is over 1\% of the maximum brain sensitivity for any source--detector pair in the 1-mm resolution. This criterion was adapted from~\cite{maria2022imaging}. 

We perform the employed PCAs for the Jacobians with MATLAB's pca-function in both resolutions. The motivation for also considering the 2-mm resolution was to reduce the number of variables for the PCA, which is particularly useful to speed up Gibbs sampling and the simulation of the approximation error noise for the BAE method. The variances in the principal directions are used to form a diagonal covariance matrix $\Gamma_{2}$ for $y$. The Jacobian of the target neonate is excluded from the PCAs, and the PCAs are limited to voxels that are within the FOV of at least one of the 214 registered models \textit{and} inside the target head, corresponding to $n \approx 2 \cdot 10^5$ unknown absorption values. In other words, we assume that the exact shape of the target head is known at voxel-level. However, the atlas FOV does not cover regions that are not inside any of the registered models.

\subsection{Noise}
\label{sec:noise}
The components of the additive measurement noise are assumed to be mutually independent, leading to a diagonal noise covariance. The standard deviation (SD) of the noise in the absolute amplitude ($\sigma_{\sf A}$) was calculated from the simulated values for the amplitude and total detected weight, corresponding to the intensity from a continuous-wave light source, as
\begin{equation}
\label{eq:ANoise}
\sigma_{\sf A} = 
\left\{
\begin{array}{ll}
6.5 \times 10^{-7}\ \sqrt{I}, & \ \text{if\ } I \leq 1.2 \times 10^{-5}, \\[1mm]
1.9 \times 10^{-4}\ {\sf A}, &\ \text{otherwise.}
\end{array}
\right.
\end{equation}
The SD of noise added to the absolute phase ($\sigma_{\varphi}$) was approximated as 
\begin{equation}
\label{eq:PhaseNoise}
\sigma_{\varphi} = 
\left\{
\begin{array}{ll}
\dfrac{1.3 \times 10^{-4}\, \degree}{\sqrt{I}}, 
&\  \text{if\ } I \leq 1.2 \times 10^{-5}, \\[3mm]
0.038\degree, & \ \text{otherwise.}
\end{array}
\right.
\end{equation}
The resulting signal-to-noise ratio (SNR) for the amplitude was between 171 and 5300 and the phase noise was between 0.038$\degree$ and 1.2$\degree$ for SDS from 5.4\,mm to 40\,mm. Since the SNR for amplitude is relatively high, we used a linear approximation to compute the SD of the additive noise in the log-amplitude measurements (see~\cite[Eq.(18)]{mozumder2024diffuse}),
\begin{equation*}
    \sigma_{\log {\sf A}} \approx \dfrac{\sigma_{\sf A}}{{\sf A}},
\end{equation*}
where the error in the approximation at the smallest SNR is of the order 1/171.

The SDs for difference data were obtained by multiplying the initial SDs by $\sqrt{2}$, which is justified by the assumed independence of different noise realizations. Finally, we assumed that there are 30 independent repetitions of the measurements that can be averaged, which results in a reduction of the SDs by a factor of $\sqrt{30}$. As a result, we get the following formulas for the SD of the noise in the difference data
\begin{equation}
\label{eq:NoiseDiffData} 
\left\{
\begin{array}{l}
\sigma_{\Delta \log A} = 0.26 \times \dfrac{\sigma_{\sf A}}{{\sf A}}, \\[1mm]
\sigma_{\Delta \varphi} = 0.26 \times \sigma_{\varphi},
\end{array}
\right.
\end{equation}
which are used for simulating the zero mean additive Gaussian measurement noise and for forming the noise covariance matrix $\Gamma_1 \in \R^{2l \times 2l}$.

\subsection{Absorption Perturbation Prior}
\label{sec:abs_prior}
Previous choices for the inverse of the prior covariance matrix $\Gamma_3$ include a discrete approximation of the negative Laplacian giving second degree smoothening~\cite{maria2022imaging}, a simple diagonal matrix corresponding to minimum norm regularization~\cite{cooper2012validating} and a combination of the two~\cite{Heiskala2012,Conference2023}. In this study, we model the covariance to decrease exponentially with increasing voxel separation 
\begin{equation*}
\big(\Gamma_3^{\mathrm{temp}}\big)_{ij} =  \sigma^2  \exp \! \bigg(-\frac{\| z_i - z_j\|^2}{2\, c_l^2} \bigg) ,  \qquad i,j=1, \dots , n,
\end{equation*}
and further restrict its spread through
\begin{equation*}
\label{eq:Gamma_trunc}
(\Gamma_3^{\mathrm{trunc}})_{ij} = 
\left\{
\begin{array}{ll}
\big( \Gamma_3^{\mathrm{temp}} \big)_{ij}\, , & \  \big( \Gamma_3^{\mathrm{temp}} \big)_{ij} > (0.01 \, \sigma)^2\, , \\[1mm]
0\, , &\    \mathrm{else},
\end{array}
\right.
\end{equation*}
to reduce the computational burden via sparsity. Since such thresholding does not in general preserve positive definiteness, we add a small diagonal term and use
\begin{equation*}
\label{eq:Gamma_final}
    \Gamma_3 = \Gamma_3^{\mathrm{trunc}} + \eta I,
\end{equation*}
with the aim to choose $\eta>0$ as small as possible compared to $\sigma^2$. The positive definiteness of $\Gamma_3$ was verified for the employed discretizations.

Here, $\sigma$ is the voxel-wise standard deviation of the expected absorption perturbation within the FOV, $z_j$ is the midpoint of voxel $j$, and $c_l$ is the correlation length that controls how fast the spatial covariance fades. The selected values for the model parameters are $\sigma$ = 0.003\, mm$^{-1}$ (vs.\ maximum simulated contrast of 0.008\, mm$^{-1}$), and $c_l$ = 3\,mm (vs.\ simulated perturbation radii of 4--6\,mm).  A benefit of the introduced model is that the parameters are easy to interpret. For validation purposes, we sampled the corresponding Gaussian distributions in 1- and 2-mm resolutions to verify that the prior indeed predicts perturbations with reasonable volume and contrast.

The independence of the primary unknown (the voxel-wise absorption changes) and the PC coefficients in \eqref{eq:priors} is motivated by the fact that the variation in the Jacobians over the group of neonates is independent of the (simulated) activity in the brain of one subject (the prior model does not recognize tissue boundaries).

\subsection{Perturbed Models and Data for Numerical Experiments}
\label{sec:perturbation_patterns}
The perturbations, i.e., the regions where the absorption coefficient was modified from the baseline values in Table \ref{table:param}, were selected automatically for each target neonate starting from the midpoints between the seven source--detector pairs visualized in Fig.~\ref{fig:probe_pert}a. For each optode pair, we first project their mid-point to the exterior boundary. This point is used as the center of a half-sphere forming an extracerebral perturbation with radius of 5\,mm, but limited by the thickness of the combined scalp--skull layer, boundary curvature and exclusion of the most exterior boundary voxels. The center point is then projected 2.5\,mm deeper than the brain surface along the surface normal. The first brain perturbation consists of the GM and WM voxels within roughly the deeper half of a sphere with center at the projected location and radius of 6\,mm. The second brain perturbation is obtained by repeating the process at a starting point 10\,mm anterior to the original midpoint and using a radius of 5\,mm for the half-sphere, creating in total one extracerebral and two brain perturbations for each of the seven source--detector pairs. The resulting perturbations at location 2 are visualized in Fig.~\ref{fig:probe_pert}b. Different multi-perturbation patterns were created by activating different combinations of the 21 possible regions.

The increase in the absorption coefficient in the brain is set to 8\,$\mathrm{m}^{-1}$, which corresponds to an increase of 40.8\,$\mu$M in the HbT concentration at 798\,nm~\cite{cope1991development}, and a 17\% contrast compared to the baseline absorption coefficient of GM in Table \ref{table:param}. However, at locations 3 and 7 in Fig.~\ref{fig:probe_pert}, we leave the second brain perturbation on the surface of the cortex and consider a smaller, more realistic contrast of 11\% or 6\,$\mathrm{m}^{-1}$~\cite{dunn2005spatial,heiskala2009probabilistic,heiskala2009significance}, which corresponds to an HbT increase of 30.6\,$\mu$M at 798\,nm~\cite{cope1991development}. Note that in reality the contrast in WM is lower than in GM~\cite{fraser2012white,spencer2025mapping}. The increase in the absorption coefficient in the scalp--skull layer is also set to 6\,$\mathrm{m}^{-1}$ to simplify the visualizations. We included extracerebral perturbations mainly to test
the algorithms --- they are of mathematical rather than physiological interest.

We use two methods to generate the difference data $b$ in \eqref{eq:inv_prob} for the described perturbation patterns: First, to test the algorithms without linearization error, we use the same linear model \eqref{eq:inv_prob} as for solving the inverse problem.  Secondly, we use the actual (nonlinear) differences between simulated log-amplitude and phase values according to \eqref{eq:MCXforwdata} with and without the absorption perturbations. Since MCX implements the Microscopic Beer-Lambert law~\cite{Fang2009,sassaroli2012equivalence,yao2018replay,Hirvi_2023}, the simulated trajectories of photon packets are independent of the absorption coefficients and we can compute the baseline and absorption-perturbed forward data from the same region-wise paths for the detected photon packets. This reduces the computational load and excludes extra stochastic noise from the difference data, making the model for additive noise described in Section \ref{sec:noise} exact, especially since we also compute the Jacobians from the same forward data for each neonate. Naturally, the latter does not hold if the Jacobian is computed in another head model or as the atlas mean.  

\begin{figure}[t!]
    \centering
    \includegraphics[width=\textwidth]{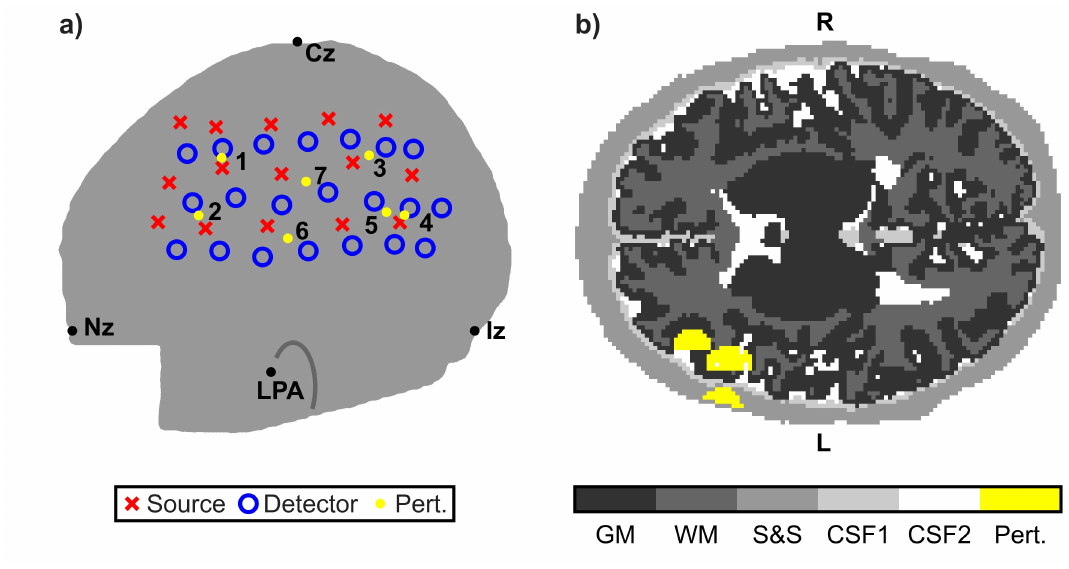}
    \caption{a) The 15 sources (red crosses), 21 detectors (blue rings) and seven perturbation sites (yellow circles; 1--7) on the reference neonate. Nz = nasion, Cz = vertex, Iz = inion, LPA = left pre-auricular point. b) Perturbations at location~2. GM = gray matter, WM = white matter, S\&S = scalp and skull, CSF = cerebrospinal fluid.}
    \label{fig:probe_pert}
\end{figure}

\subsection{Evaluating the Reconstructions}
\label{sec:metrics}
In addition to visual inspection of the reconstructions, we consider two error measures: (i) the contrast-to-noise-ratio (CNR), and (ii) the root-mean-squared-error (RMSE; shortened as E) of the difference between the target absorption perturbation map and the reconstruction in the whole FOV. We modified the definition for CNR in~\cite{nissila2006comparison} by considering the mean reconstructed value over multiple perturbed regions instead of a single maximum value:
\begin{equation}
    \label{eq:CNR_def}
    {\rm CNR} = \dfrac{\langle x\rangle_{\rm pert} - \langle x \rangle_{\rm unpert}}{\sigma_{\rm unpert}}\, ,
\end{equation}
where $ \langle x  \rangle $ is the mean reconstructed value in the indicated region (perturbation or background) and $\sigma_{\rm unpert}$ is the SD of the reconstructions in the background.

\section{Results}
\label{sec:numerics}

\subsection{Registration}
The average registration accuracy indicates the suitability of the Jacobian atlas for modeling the target neonate. We defined the registration error as the mean distance between the registered and the target head surface mesh under the measurement probe, using the dense meshes introduced in Section \ref{sec:optodes}. For the reference neonate, which was the actual registration target in Section \ref{sec:registration}, the median registration error for all 214 neonates is 1.36\,mm (range 0.62--2.64\,mm). If we only consider the mean distance between the optodes, the median error is 1.38\,mm (range 0.60--2.69\,mm). The registration accuracy is not uniform, but weakens especially towards the posterior edge of the imaged region on the reference, increasing the mean errors. Fig.\ \ref{fig:registration} shows the median registration errors when we consider each of the registered neonates as the target instead of the actual reference. The registration accuracy is not the best for the actual reference, since other neonates have mutually more similar original head shapes.

% Average registration errors (distance to mesh/node-wise)
% N38: Mean 2.2 mm, Median 2.2
% N52: Mean 2.4 mm, Median 2.4 mm, 
% N71: Mean 2.5 mm, Median 2.5 mm 
% N173: Mean 1.4 mm, Median 1.3 mm - ok!

% 23.2.2026
% Average registration errors (distance to mesh/node-wise)
%   N4:  Mean 1.37 mm, Median 1.32 mm
%  N46:  Mean 1.27 mm, Median 1.17 mm, 
%  N98:  Mean 1.48 mm, Median 1.39 mm 
% N158:  Mean 1.29 mm, Median 1.25 mm
% N188:  Mean 1.40 mm, Median 1.36 mm
\begin{figure}[t!]
    \centering
    \includegraphics[width=0.95\textwidth]{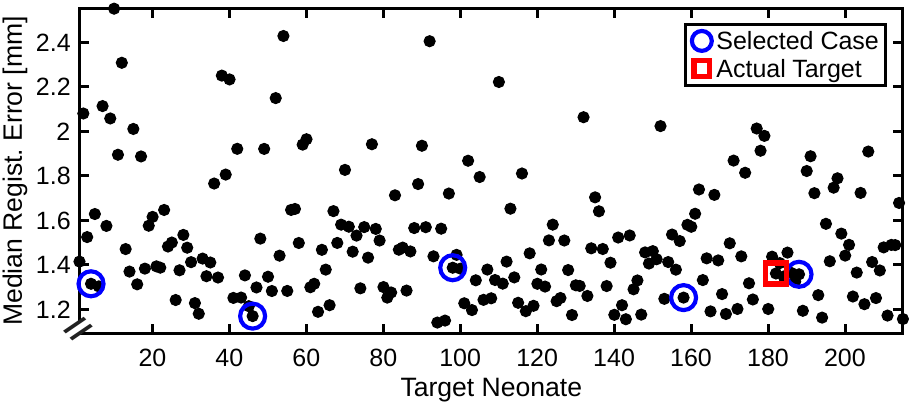}
    \caption{The median of registration errors for the 214 other neonates in the atlas considering one registered neonate head as the target (horizontal axis, ordered by age). The error for each neonate pair is the mean distance between the registered surface meshes under the measurement probe. The actual common target (red square) in the registration was number 182. Blue circles refer to the selections in Table \ref{table:cases}.}   
    \label{fig:registration}
\end{figure}
    
\subsection{Representing Jacobians in the Principal Component Basis}

The ability of a relatively low number of PCs to represent the target Jacobians with reasonable accuracy is another prerequisite for the justifiability of our approach. Here we study this aspect for the row-wise PCs used in connection to the BCD, G--N and Gibbs algorithms; there is no such need to investigate the approximation properties of the Frobenius-based PCs since all of them are used when approximating the statistics for the approximation error noise, cf.~Section~\ref{sec:BAE}. As the first step, we visualized a number of first PCs over the FOV for some source--detector pairs, i.e., some rows of the Jacobian. These PCs manifest themselves as smooth perturbations in the sensitivity profile despite the low number of samples (Jacobian matrices for different neonates) compared to the number of unknowns (voxels). The complexity of the patterns increases with decreasing variance. 

A strength of PCA is that a low number of PCs is often sufficient for explaining the majority of the variation in the data. We limited to using the first ten PCs for each measurement, since this was evaluated as a reasonable compromise between the representation power of the PCs and their reliability considering the low sample size. Moreover, we observed that our algorithms mainly utilize only the first couple of PCs, so employing only ten PCs per row decreases the computational load without clear effects on the reconstruction quality.

To measure the ability of the first ten row-wise PCs to represent the rows of the Jacobians, we estimated the relative error in the Euclidean norm as $\| \hat{a} - a \| / \| a \|$, where $\hat{a}$ is the optimal 10-PC representation for a row $a$ in a true Jacobian. Fig.~\ref{fig:pc_reco_relat_error} visualizes the average errors over all neonates and all measurements for the log-amplitude~(a) and phase Jacobians (b) with  voxel side length of 2\,mm. The squares represent the mean and the vertical lines the standard deviation of the error for the measurements/rows for which the SDS is within the interval indicated on the horizontal axis. The relative errors are smaller for the log-amplitude than phase Jacobians, and they increase as a function of the SDS in both cases. With log-amplitude and short-channel measurements, the most sensitive voxels are confined to a smaller volume closer to the surface compared to phase and long-channel measurements, which may explain the smaller errors. In 1-mm resolution, the relative errors are somewhat higher but the general trends are the same. The sufficiency of the PC representation is ultimately evaluated via the performance of the reconstruction algorithms in the numerical experiments. 

% 5.1.2026 AH
\begin{figure}[t!]
    \centering
    \includegraphics[width=0.95\linewidth]{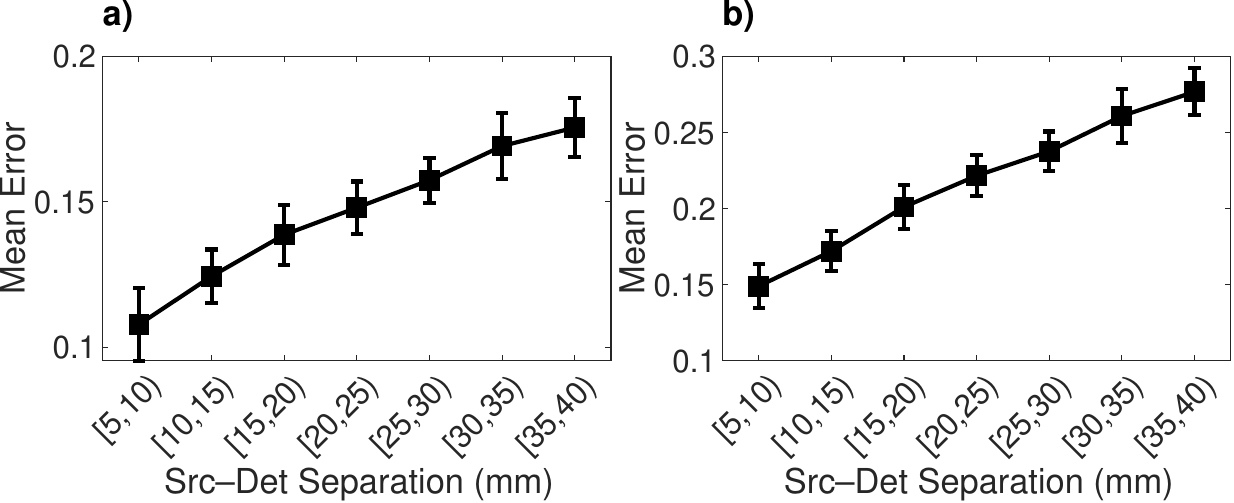}
    \caption{Statistics for the relative error of the PC representation for the rows of the Jacobians in the 2-mm resolution with the first 10 row-wise PCs. (a)~log-amplitude. (b)~phase. The squares depict the mean relative errors and the vertical lines their standard deviations for rows in the SDS ranges indicated on the horizontal axis.}
    \label{fig:pc_reco_relat_error}
\end{figure}

\subsection{Numerical Experiments}
We consider four multi-perturbation patterns and five neonates that are circled in Fig.\ \ref{fig:registration}. The results are presented for the five settings described in Table \ref{table:cases} and referred to as Cases 1--5. We selected three full term, one late preterm and one very preterm neonate with differing baseline optical parameters (see Table \ref{table:param}) based on the combined registration accuracy and CNR improvement. The original head circumferences were 26.5--35.9\,cm, before all neonates were registered to the reference head model with circumference of 33.9\,cm. The median and minimum of mean registration errors in Table \ref{table:cases} indicate the overall goodness of the atlas and the closest reference for each target neonate, respectively. As described earlier, the target's Jacobian is always excluded from the PC basis and thus not directly exploited in the reconstruction algorithms. However, we assume that the head shape of the target is known exactly when limiting the FOV.

Table \ref{table:cases} lists the locations of the perturbations using the labels in Fig.\ \ref{fig:probe_pert}. In Case~1, the perturbation map contains two perturbations in the brain, in Case 2 there are four, and the rest contain ten perturbations in the brain. In Cases 4--5, there are also five perturbations on the head surface. As described in Section \ref{sec:perturbation_patterns}, the absorption increase in the scalp perturbations is 6\,m$^{-1}$ and in the brain 8\,m$^{-1}$, except for the cortical perturbations at locations 3 and 7 in Fig.\ \ref{fig:probe_pert}, where the increase is 6\,m$^{-1}$.

As described at the end of Section \ref{sec:perturbation_patterns}, we employed two spatial resolutions for the reconstructions and two approaches to generate the difference data for each perturbed model. In Table \ref{table:cases}, if ``Resolution'' is 1\,mm, we used the actual simulated difference data and the absorption change was reconstructed in 1-mm resolution. If ``Resolution'' is 2\,mm, the data was generated directly with the linear model using the true Jacobian of the target neonate in 1-mm resolution, and the absorption change was reconstructed in 2-mm resolution. If ``Resolution'' is ``Both'', we used actual simulated difference data and computed the reconstructions in both resolutions. We tested all Algorithms~\ref{alg:alternating}, \ref{alg:GN} and \ref{alg:Gibbs} (G--N, BCD, Gibbs) as well as the BAE method for Case~5. Only the G--N algorithm was applied to Cases 1--4.

\begin{table}[t!]
\centering
\caption{The five considered combinations (Cases 1--5) of target neonate and absorption perturbation pattern from Section \ref{sec:perturbation_patterns}. The combined gestational and chronological age (gwk) and the original head circumference are provided for each neonate. Registration error is given as the median/minimum of the mean registration errors over all 214 registered neonates. The PCA and reconstruction are computed either in 1- or 2-mm resolution.}  
\begin{tabular}{ m{0.8cm} m{2cm} m{1.1cm} m{1.8cm} m{1.8cm} m{0.9cm} m{1.6cm} }
 \toprule
 Case & Neonate & Regis-\newline tration Errors [mm] & Scalp--Skull Perturbations\newline (Fig.\ \ref{fig:probe_pert}) & Brain Per-\newline turbations (Fig.\ \ref{fig:probe_pert}) & Reso-\newline lution [mm] & Result\newline Algorithm \\ 
 \midrule
  1 & 39.0\,gwk/ & 1.39/ & - & 7 & 1 & Figs.\ \ref{fig:collection_1mm}a--c \\
    & 34.2\,cm & 0.71 & & & & G--N \\ % N98, P21
  2 & 35.4\,gwk/ & 1.17/ & - & 1, 3 & 1 & Figs.\ \ref{fig:collection_1mm}d--f \\
    & 32.2\,cm & 0.46 & & & & G--N \\ % N46, P50
  3 & 40.9\,gwk/ & 1.25/ & - & 1, 2, 3, 5, 7 & 2 & Figs.\ \ref{fig:linear_2mm_collection}a--c \\
    & 35.4\,cm & 0.56 & & & & G--N \\ % N158, P49
  4 & 29.9\,gwk/ & 1.32/ & 1, 2, 3, 4, 7 & 1, 2, 3, 4, 7 & 2 & Figs.\ \ref{fig:linear_2mm_collection}d--f \\
    & 26.5\,cm & 0.56 & & & & G--N \\ % N4, P56
  5 & 41.9\,gwk/ & 1.36/ & 1, 2, 3, 4, 7 & 1, 2, 3, 4, 7 & Both & Fig.\ \ref{fig:comparison_2mm} \\	 
    & 35.9\,cm & 0.58 & & & & All \\ % N188, P56
 \bottomrule	
\end{tabular}
\label{table:cases}
\end{table}

\subsubsection{Gauss--Newton Algorithm}
The results produced by the G--N algorithm (Algorithm~\ref{alg:GN}) for Cases 3--4 are presented in Fig.\ \ref{fig:linear_2mm_collection}, and for Cases 1--2 in Fig.~\ref{fig:collection_1mm}. The G--N algorithm was run for 100 rounds with $\kappa = 0.2$, which was evaluated to suffice for convergence in all cases. In Figs.~\ref{fig:linear_2mm_collection} and \ref{fig:collection_1mm}, the black line marks the borders of the target absorption perturbations, and each row corresponds to one case. Panels (a, d) show the reconstructions obtained with the true Jacobian of the considered neonate; see Remark~\ref{remark:fixed}. The reconstructions obtained with the mean Jacobian are presented in (b, e), and the reconstructions by the G--N algorithm in (c, f). Fig.\ \ref{fig:linear_2mm_collection} visualizes mean reconstructions over axial slices covering all perturbed regions. Fig.~\ref{fig:collection_1mm} considers a single axial slice at the mid-height of the perturbed region, since there are only one or two nearby brain perturbations approximately at the same height. The colormap is constructed separately for each reconstruction. The visualization threshold is set to 1\% of the maximum absolute value in each image. For clarity, the background image only separates the combined scalp and skull layer, the subarachnoid and the brain (combined GM, WM and sulci). 

As expected, the true Jacobian of the target neonate yields the best reconstructions in all cases. The regions with maximum contrast follow the target perturbation boundaries relatively well without spreading the activity towards the exterior surface. The reconstructed contrast is approximately 18\% lower than the target, which cannot be considered atypical for DOT. All in all, the quality of the reconstructions supports the suitability of the absorption perturbation prior introduced in Section \ref{sec:abs_prior}.

The accuracy metrics, i.e., the CNR from~\eqref{eq:CNR_def} and the root-mean-squared-error~E, are presented in the titles of the reconstruction images. In each case, the G--N algorithm manages to increase the CNR and decrease the error E when compared to the reconstruction obtained with the mean Jacobian. However, the metrics are not as good as for the true Jacobian. The improvement compared to the mean Jacobian can also be verified visually. The G--N algorithm locates the absorption perturbations more accurately, although their shapes are blurred and not entirely confined within the black lines. Thus, the algorithm seems to perform better at focusing on the correct region of interest than separating nearby perturbations or reconstructing their exact shape. The background is also smoother, with less spatial undulation. However, in Fig.~\ref{fig:linear_2mm_collection}f, the G--N algorithm fails to reconstruct the absorption increase on the head surface, and places a smooth negative region on the surface instead. Such a negative layer is a recurring phenomenon in G--N reconstructions, though it is less apparent in the 1-mm cases considered in Fig.~\ref{fig:collection_1mm}. The G--N algorithm produces the lowest and the mean Jacobian the highest contrasts for the reconstructed perturbations, but the latter comes with the price of severe oscillations in the background, as also indicated by the CNR and~E.

The G--N algorithm did not find the correct $y$-coefficients, i.e., the ones giving the best representation for the true Jacobian in the PC basis, in any of the presented cases (the same also applies to BCD and the Gibbs sampler below). At best, the coefficients of the first PCs had correct signs and roughly correct magnitudes for some optode pairs close to the perturbed regions. This happened especially if we had only one superficial perturbation, but these cases have less practical relevance and are not considered further. This supports the interpretation that the PC coefficients are not identifiable physical parameters in the present setting. Nevertheless, it seems that rather than enabling the reconstruction of the target Jacobian, incorporating PCs of the Jacobian rows in the algorithm provides extra flexibility that allows to reduce the most severe reconstruction artifacts caused by the mismodeling of the target anatomy.

% IN THIS FIGURE (28.7.2026):
% – 1st row: 158, perturbation pattern 49
% – 2nd row: N4, perturbation pattern 56
\begin{figure}[t!]
    \centering
    \includegraphics[width=0.9\linewidth]{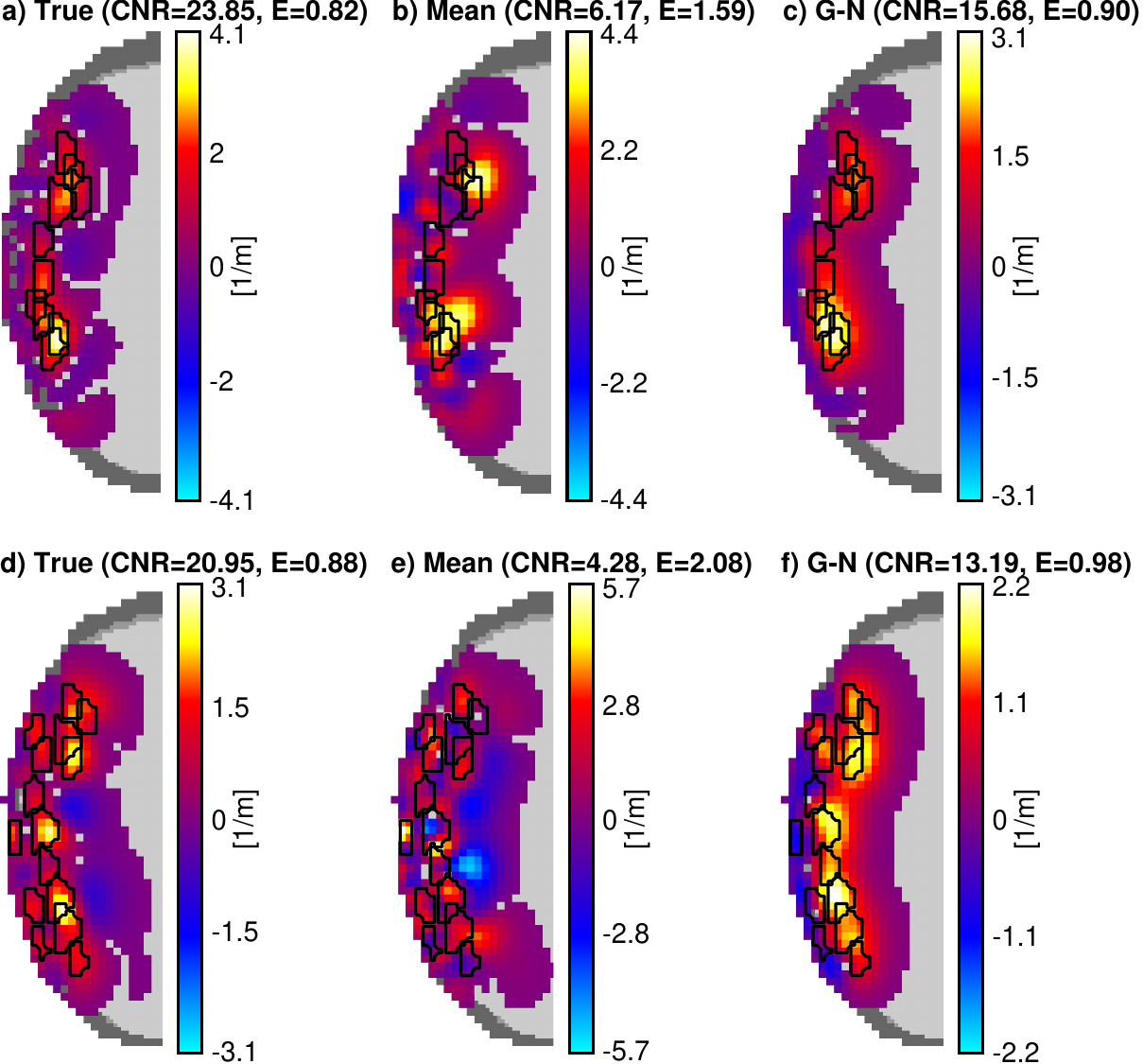}
    \caption{First row: Case~3. Second row: Case~4. Reconstructions averaged over axial slices covering the perturbed region, computed with the true Jacobian of the target neonate (a,~d), the mean Jacobian over the other 214 neonates registered to the target head shape (b,~e), and by improving the rows of the mean Jacobian along their first 10 PCs with the G--N algorithm (c,~f). The black lines mark the target perturbations with contrast of 6\,m$^{-1}$ or 8\,m$^{-1}$ (30.6\,$\mu$M or 40.8\,$\mu$M at 798\,nm, respectively).}
    \label{fig:linear_2mm_collection}
\end{figure}

% IN THIS FIGURE (29.7.2026):
% – 1st row: N98, perturbation pattern 21 
% – 2nd row: N46, perturbation pattern 50
\begin{figure}[t!]
    \centering
    \includegraphics[width=\textwidth]{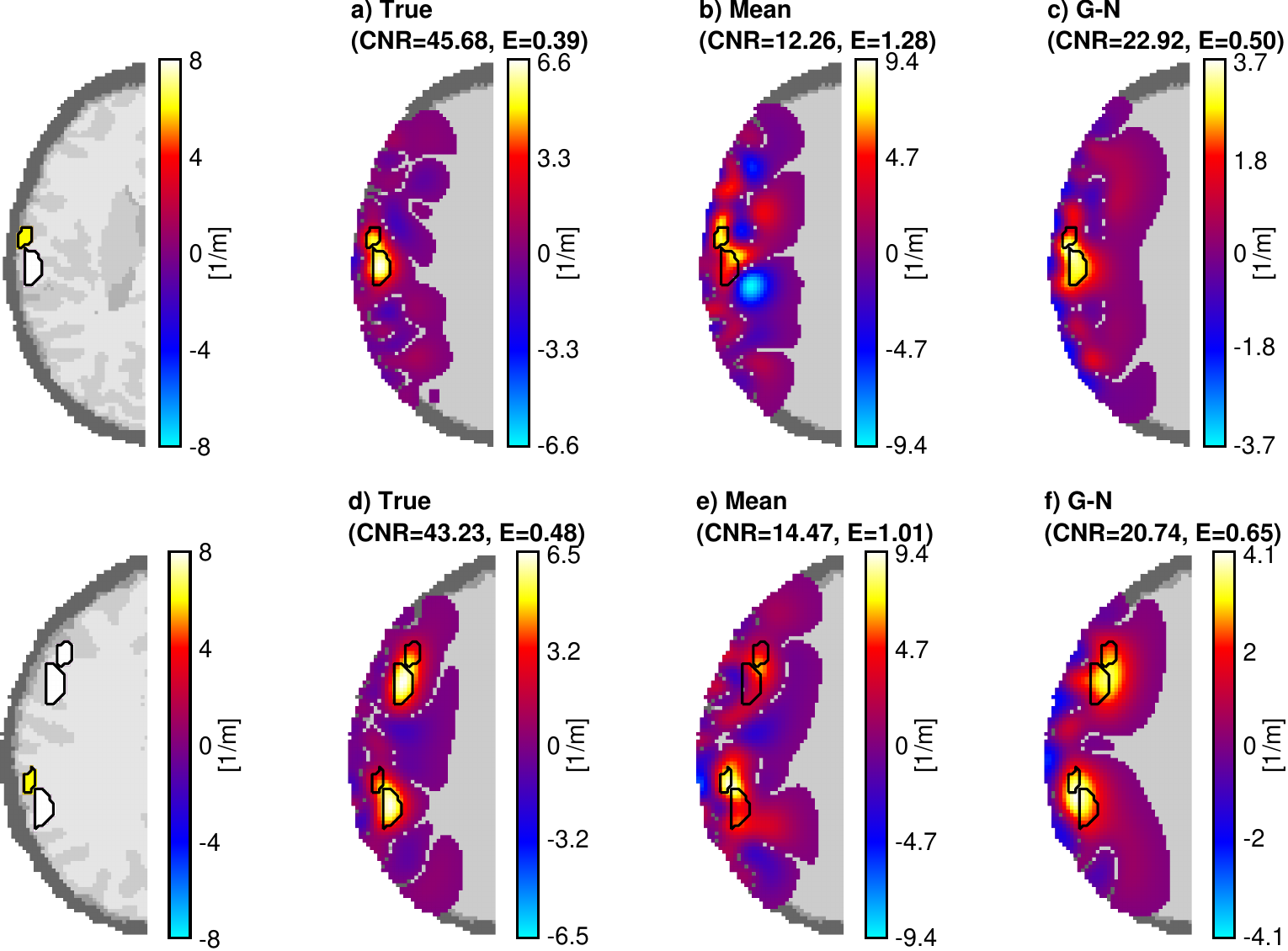}
    \caption{First row: Case~1. Second row: Case~2. Three-dimensional reconstructions restricted onto a single axial slice at the mid-height of the perturbed regions, computed with the true Jacobian of the target neonate (a,~d), the mean Jacobian over the other 214 registered neonates (b,~e), and by improving the rows of the mean Jacobian along their first 10 PCs with the G--N algorithm (c,~f). The black lines mark the target perturbations with contrast of 6\,m$^{-1}$ or 8\,m$^{-1}$. The target absorption perturbation is presented in the leftmost column with tissue types visualized in the background. 
    }
    \label{fig:collection_1mm}
\end{figure}

To further verify the performance of the G--N algorithm, we considered each of the 215 neonates consecutively as the target and computed the results for all four perturbation maps described in Table~\ref{table:cases}. The difference data was simulated in the 1-mm resolution, and the PCA and reconstructions were formed in the 2-mm resolution. Fig.\ \ref{fig:linear_2mm_cnr_collection} presents the CNR values for the reconstructions produced by the true Jacobian (black circles), the atlas mean Jacobian (red crosses), and the G--N algorithm (blue dots). The horizontal axis corresponds to the target neonates sorted according to the CNR of the G--N reconstruction. Figs.\ \ref{fig:linear_2mm_cnr_collection}a, b, c and d correspond to the perturbation patterns in Cases 1, 2, 3 and 4 \& 5, respectively. On average, the G--N algorithm manages to increase the CNR, but not up to the levels obtained with the true Jacobian. The algorithm also appears to work better when there are more perturbations at different locations and/or perturbations also on the head surface: in Figs.\ \ref{fig:linear_2mm_cnr_collection}a, b, c and d, the total number of perturbed regions are 2, 4, 10 and 15, respectively, and the relative performance of the algorithm clearly improves as a function of this number. In some tests, the mean Jacobian managed to separate neighboring perturbations close to the level of the true Jacobian, whereas Gauss--Newton blurred nearby perturbations into one with reduced (overall) contrast, which presumably explains some of the cases in Fig.\ \ref{fig:linear_2mm_cnr_collection}, where Gauss--Newton lowers the CNR compared to the mean Jacobian.
 
% IN THIS FIGURE (28.7.2026): CNR PLOTS
% - 1st row: perturbation patterns 21, 50
% - 2nd row: perturbation patterns 49, 56
\begin{figure}[t!]
    \centering
    \includegraphics[width=0.9\linewidth]{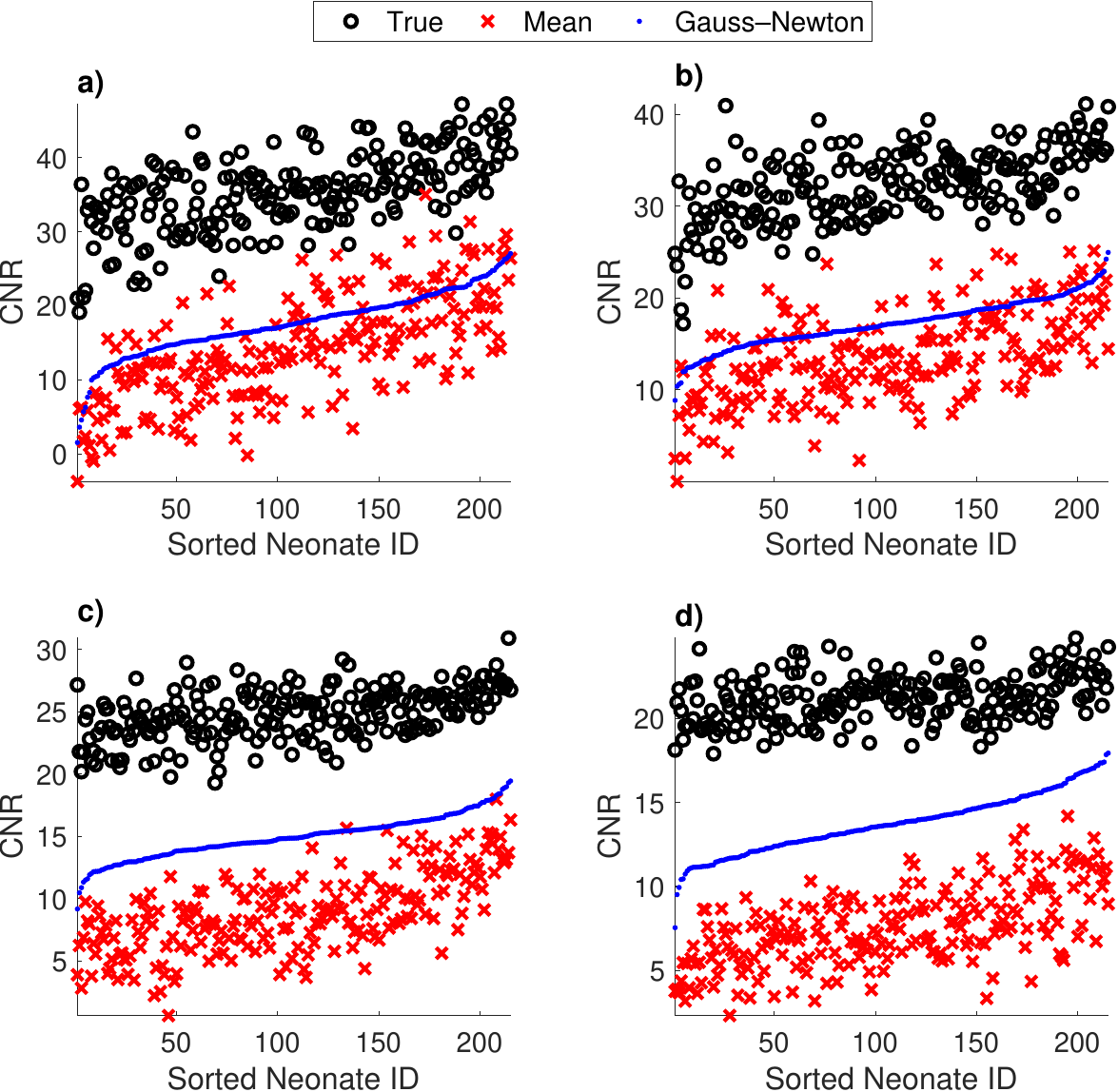}
    \caption{CNRs for reconstructions for (a) Case 1, (b) Case 2, (c) Case 3 and (d) Case 4 \& 5 in Table \ref{table:cases} with each of the 215 neonates consecutively as the registration target. The black circles correspond to the true Jacobian of the target, the red crosses to the mean Jacobian over the atlas, and the blue dots to the G--N algorithm. In each case, the target Jacobian was excluded from the computation of the mean Jacobian and the PCs used by the G--N. The neonates are sorted  by the CNRs for the G--N.}
    \label{fig:linear_2mm_cnr_collection}
\end{figure}

\subsubsection{Comparison of Different Algorithms}
Finally, we also test the other three algorithms, namely, BCD (Algorithm~\ref{alg:alternating}), the Gibbs sampler (Algorithm~\ref{alg:Gibbs}), and for comparison the BAE method (Section~\ref{sec:BAE}) using Case 5 from Table \ref{table:cases}. The difference data for the model with 15 perturbations was simulated in the 1-mm resolution. However, we ran the BCD and  Gibbs sampler and computed the BAE reconstruction in the 2-mm resolution since the former two required high number of iterations for convergence compared to the G--N algorithm and we preferred the lower resolution for the BAE method to allow easily drawing a sufficient sample for \eqref{eq:sample_cov}. 

The reconstructions are presented in Fig.\ \ref{fig:comparison_2mm} for one axial slice close to the mid-height of the imaged region shown in Fig.\ \ref{fig:probe_pert}, with panel (a) showing the target perturbation. The reconstruction obtained with the true Jacobian of the target is depicted in panel (b), the one obtained with the atlas mean Jacobian in (c), and the one obtained with the G--N algorithm in (d), all in the 1-mm resolution. The reconstruction by the G--N algorithm in the 2-mm resolution is presented in panel (e), and the reconstruction after 7\,000 iterations of BCD in panel (f). Panel (g) shows the mean of the first 100\,000 samples of the Gibbs sampler and panel (h) the reconstruction by the BAE method with $N=100\,000$ in \eqref{eq:sample_cov}. The accuracy metrics, i.e., the CNR and the error E, are again presented in the title of each subfigure. Although we only provide one axial slice to support the following observations, other reconstruction slices through the perturbed region exhibited similar characteristics, with the selected slice not presenting the ``best-case scenario'' for our algorithms.

All algorithms managed to visually improve the reconstruction by the mean Jacobian in Fig.\ \ref{fig:comparison_2mm}c, especially around the two top-most and deepest brain perturbations. The mean Jacobian also yields the lowest CNR and the highest reconstruction error~E, which all algorithms manage to improve, though not to the level of the true Jacobian. As in Figs.~\ref{fig:linear_2mm_collection} and \ref{fig:collection_1mm}, the reconstruction with the mean Jacobian clearly exhibits the largest dynamic range but also provides by far the worst localization of the perturbed regions. On the contrary, Algorithms~\ref{alg:alternating}, \ref{alg:GN}, \ref{alg:Gibbs}, and the BAE method reconstruct the target perturbations with a 30--50\% lower contrast than the true Jacobian, but they also significantly reduce the background oscillations compared to the mean Jacobian.

We do not observe clear differences in the performance of the G--N algorithm between the 1-mm and 2-mm resolution, though the PC explanation ratios are higher in the latter case. The 2-mm reconstructions look similar independent of whether actual simulated difference data (Fig.~\ref{fig:comparison_2mm}) or data generated directly by the linear model (Fig.~\ref{fig:linear_2mm_collection}) were used, which supports the sufficiency of the linear approximation.

The G--N algorithm (both resolutions) and BCD produced qualitatively similar reconstructions, even though they correspond to different local minima. In particular, both algorithms create a negative layer by the exterior surface, but unlike in Fig.\ \ref{fig:linear_2mm_collection}f, the superficial perturbation is reconstructed with the correct sign in Fig.\ \ref{fig:comparison_2mm}. Both G--N and the BCD locate all perturbations relatively well, although they spread outside the target boundaries towards the surface. Interestingly, both algorithms manage to reconstruct some perturbations deeper in the brain, which are not as visible in the reconstructions by the true Jacobian, the Gibbs sampler, or the BAE method. However, this may partially be related to the selected colormap and the fact that the deeper perturbations are reconstructed with a lower contrast by the true Jacobian. 

The reconstruction by the Gibbs sampler is the closest match with the reconstruction obtained using the true Jacobian. On the visualized axial slice, the Gibbs sampler highlights the same two perturbations close to the mid-height of the slice, and the contrast is also closest to the reference provided by the true Jacobian. Compared to the G--N and BCD algorithms, the Gibbs sampler also produces higher CNR and lower reconstruction error E. The better performance of the Gibbs sampler may be explained by its ability to survey the whole joint posterior distribution without focusing its attention on a single local minimum --- on the negative side, the computational burden of the Gibbs sampler is significantly higher than that of the G--N algorithm. 

To briefly evaluate the quality of the sample produced by Algorithm~\ref{alg:Gibbs}, Fig.~\ref{fig:Gibbs_samples} presents the autocorrelation functions (b,~e) and sample histories (c,~f) for the absorption values of two voxels, one with a high perturbation in its absorption level and the other with a small one. The corresponding effective sample sizes are $49\,958$ and $95\,610$, and neither voxel-wise sample exhibits need for disregarding iterates in a burn-in sequence. Over all voxels, the effective sample size had the median of $93\,583$, and ranged from $5\,314$ to $97\,416$. According to these figures and additional tests (not documented here) on the effect that the sample size has on the mean of the Gibbs sample, the reconstruction in Fig.~\ref{fig:comparison_2mm}g should present a reasonable estimate for the CM estimate induced by the considered forward model.

The reconstruction in Fig.~\ref{fig:comparison_2mm}h by the BAE method resembles the other reconstructions, especially the one produced by the Gibbs sampler, and exhibits the best accuracy metrics. This is a recurring phenomenon in our tests: the BAE method performs approximately as well, or slightly better, than the other algorithms. As a consequence, if the statistics of the approximation error noise can be sampled, the BAE method is a recommended option, especially because it allows for a very fast online reconstruction phase.

% IN THIS FIGURE (23.2.2026)
%   – N188, perturbation pattern 56
%   - Gibbs: all 100 000 samples
%   – Alternating: 7000 iterations
\begin{figure}[t!]
    \centering
    \includegraphics[width=\textwidth]{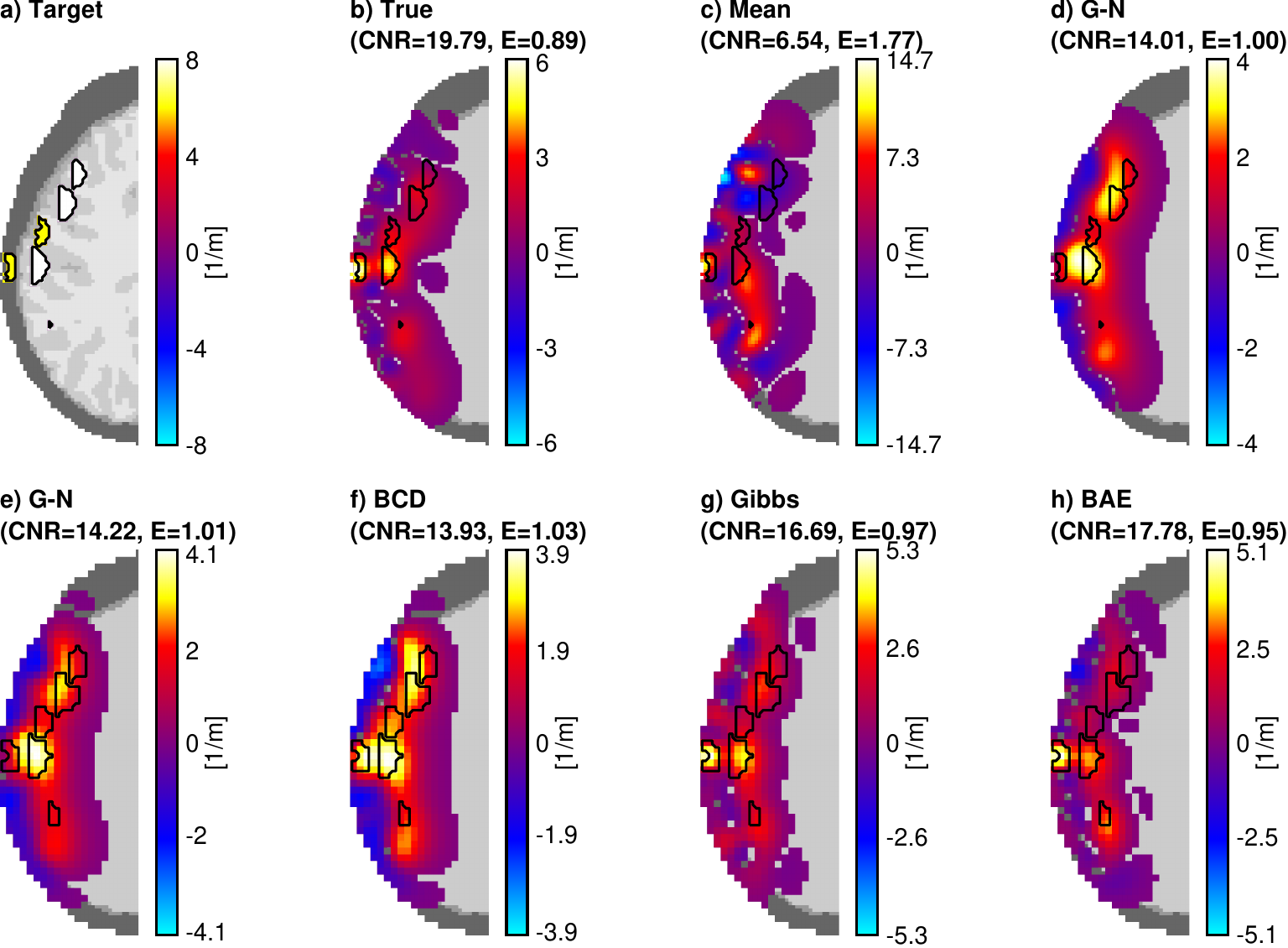} % 29.7.2026 AH
    \caption{Case 5. First row: the target (a) and reconstructions computed in the 1-mm resolution with the true Jacobian of the neonate (b), the mean Jacobian over the other 214 neonates registered to the target head shape (c), and by improving the rows of the mean Jacobian along their first 10 PCs with the G--N algorithm (d). Second row: reconstructions computed in the 2-mm resolution by the G--N algorithm (e), by taking 7\,000 iterations of BCD (f), by taking the mean of the first 100\,000 samples by the Gibbs sampler (g), and by the BAE method with $N=100\,000$ in \eqref{eq:sample_cov} (h). The black lines mark the target perturbations with contrast of 6\,m$^{-1}$ or 8\,m$^{-1}$. 
    } 
    \label{fig:comparison_2mm}
\end{figure}

\begin{figure}[t!]
    \centering
    \includegraphics[width=\textwidth]{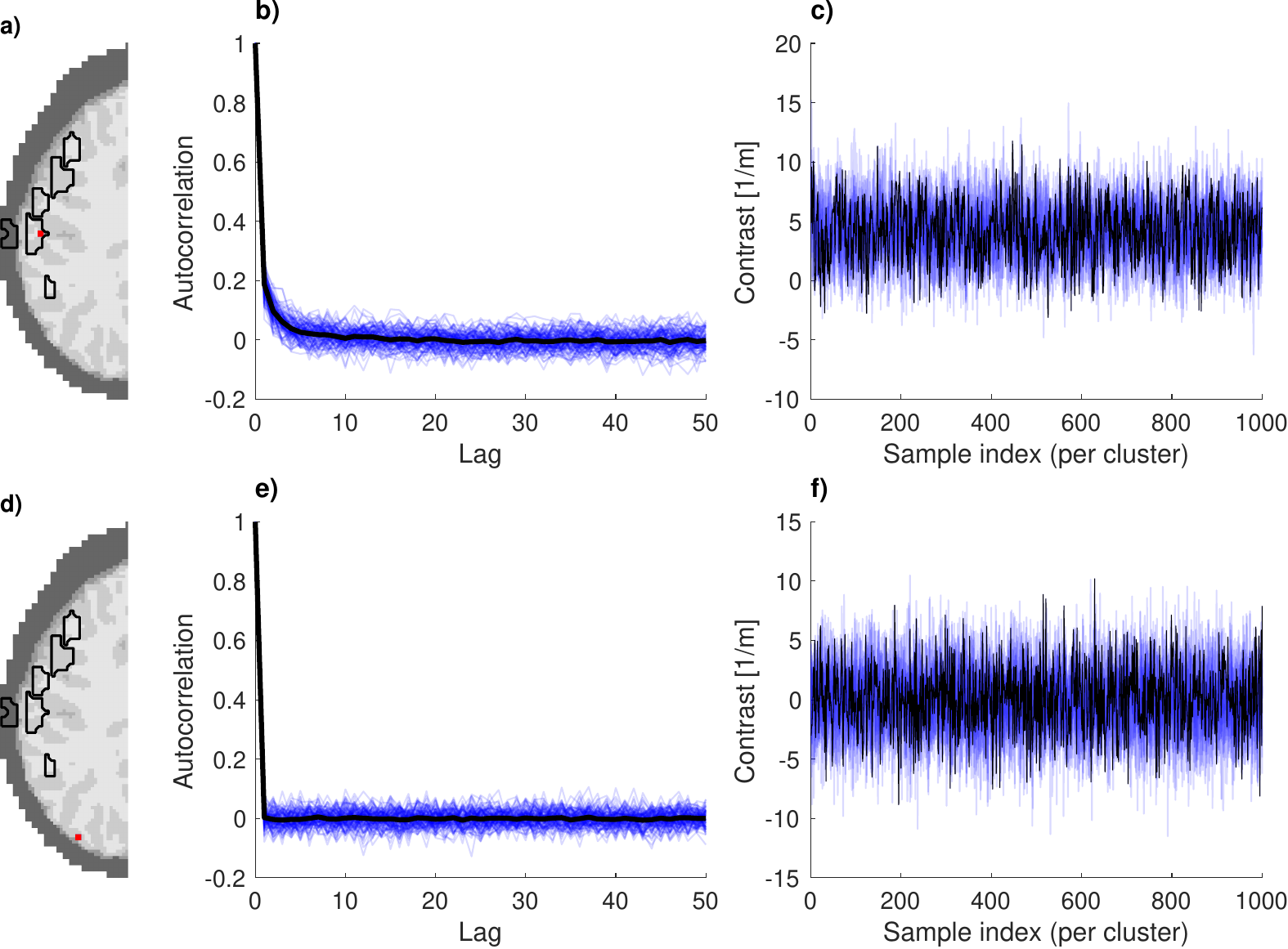}
    \caption{Voxel-wise behavior of the sample for the reconstruction in Fig.~\ref{fig:comparison_2mm}g. The sample was generated by running the Gibbs sampler in parallel across 100 computing clusters, each producing 1\,000 samples. Top and bottom rows correspond to voxels with relatively high and low reconstructed changes in the absorption coefficient, respectively. Locations of the voxels with respect to the target perturbations are visualized in red in the leftmost column. Autocorrelation functions are presented in the middle: the blue lines correspond to individual clusters, and the black line depicts the overall mean. The superposition of sample histories of 10 randomly chosen clusters is visualized on the right, with the result for one cluster highlighted in black.}
    \label{fig:Gibbs_samples}
\end{figure}

\subsubsection{Computational Resources and Requirements}

The Triton Cluster supported by the Aalto Science-IT project was used for the heaviest offline computations: simulating the measurement data and Jacobians on GPUs, performing PCA over the Jacobians, estimating the covariance of the approximation error noise, and running the Gibbs sampler. Apart from the MCX simulations on GPUs, the other preprocessing tasks could also be performed on regular computers by allowing for longer computation times or using the lower 2-mm resolution.

As the Woodbury matrix formula allows operating with inverse matrices of data dimension $2l=420$ in \eqref{eq:GN_inversion}, computing a single G--N reconstruction took about 35\,s in 1-mm resolution and 0.7\,s in 2-mm resolution in a non-optimized MATLAB implementation on a laptop with 64 GB RAM and an Apple M1 Max chip. Although Algorithm~\ref{alg:alternating} also benefits from the Woodbury matrix formula, BCD was considerably slower due to the need for almost 100 times as many iterations for convergence. The Gibbs sampler requires orders of magnitude higher computation time because one cannot avoid repeated sampling from the high-dimensional densities; the reconstruction in Fig.~\ref{fig:comparison_2mm}g was generated by running the Gibbs sampler in parallel across 100 computing clusters, each allocated 6 CPUs and producing 1\,000 samples. Per-cluster wall-clock runtime had a median of 30.4\,h, with the full run completing in the time of the slowest cluster (1\,day, 10\,h, 31\,min). In contrast, once the covariance matrix of the approximation error is available, a BAE reconstruction can be formed practically instantaneously, with cost comparable to one iteration of the G--N or BCD algorithms.

\subsection{Limitations}
\label{sec:remarks}
We observed that the G--N and BCD algorithms converge to local minima. Gibbs sampling and the BAE method seem to produce slightly better reconstructions than these optimization techniques, but they require more intensive computations: for Gibbs sampling, the online reconstruction phase is time-consuming, whereas the BAE requires forming PCs in the Frobenius topology and computing the statistics of the approximation error noise in an offline phase. The extra flexibility provided by the PCs seems to in any case enable handling some inaccuracy in the forward model, but the PC coefficients should be interpreted as regularized nuisance variables rather than identifiable physical parameters, and the way in which these algorithms utilize them remains a topic for further study. Possible future approaches for the introduced bilinear tensor formulation include the expectation--maximization algorithm, utilizing the marginalized posterior in Proposition~\ref{prop:marg_post}, tensorial lifting~\cite{Beinert19,Beinert20}, and, depending on the DOT signal analysis, re-formulating the problem for a time series of difference data, for example, by using previous time steps as initial guesses for the later ones.

The terraced voxel boundaries do not model the smoothness and curvature of the head surface realistically~\cite{yan2020}, and we were concerned that the PCA could focus on the boundary variation. We therefore repeated all analysis with a semi-hybrid model combining the voxel inner contents with a surface mesh in the Split-Voxel Monte Carlo (SVMC) software~\cite{yan2020}. The obtained results were very similar, but are not presented explicitly due to limited space.

The noise model does not consider physiological noise or covariance between the channels, and 30 successfully repeated measurements is a high requirement in practice, especially with multiple stimulus types. Temporally correlated background physiology affects how repeated measurements improve the CNR of the difference data and may cause unexpected reconstruction artifacts. Increasing the noise levels makes the reconstruction with the true Jacobian less accurate, which makes it more difficult to distinguish artifacts related to the inaccurate forward model from improvements by the algorithms. In practice, modeling errors may be partly indistinguishable under real measurement noise, which can be a challenge.

The mean Jacobian performs here poorer than what is expected from more age-specific atlas-based anatomical models in general, since we have included a wide range of gestational ages with varying success in the linear registration and the FOV is cut by the target head. Thus, the presented reconstructions by the mean Jacobian do not correspond to the expected performance of atlas-based models. At the same time, the task is more challenging for our algorithms, since the initial guess is further away from a local minimum and the PCs carry more variation. Ideally, one would have models for thousands of newborns with a maximum age difference of one week and with minimal registration error~\cite{CollinsJones2021}.

If the original volume of the head is clearly different from the reference, linear registration is not justified, for example, since the subarachnoid layer is thin in all age groups and does not increase linearly with volume. As a result, when the (extremely) preterm neonates are registered to the full term reference, the anatomies become unrealistic to some extent. However, we did not exclude any age group, since even the current number of registered models (214) was too small for the number of voxels in the FOV. In the five considered cases, the median registration errors are 1.17--1.4\,mm and the closest registered model is 0.46--0.71\,mm apart on average. Clear outlier heads can nevertheless ``break'' the mean Jacobian and cause clear errors in the corresponding reconstruction, effectively exaggerating variation in scalp--brain distances. The mean Jacobian therefore provides a non-optimal reference for testing the introduced algorithms, but does not present how average anatomical models should be created.

To conclude, we suggested a new way to utilize an atlas of head models but do not claim our approach is generally more optimal. If the atlas contains a model that is a good approximation of the target anatomy, using that individual's Jacobians can produce better images, especially contrast-wise, than the tested algorithms. Once the PC basis has been constructed, however, the G--N algorithm can quickly form a candidate Jacobian and a reconstruction without the need to search for the closest match. Moreover, if the statistics of the approximation error noise can be formed in the offline phase, a BAE reconstruction can be computed even faster. On the other hand, the offline computational load is relatively high but the process is automatic, whereas exact surface registration for a more representative average anatomy might require manual work.

\section{Conclusions}
\label{sec:conclusion}
We studied a bilinear finite-dimensional inverse problem with application to diffuse optical tomography of neonates. The aim was to tackle inaccuracies in the forward operator caused by the unknown head anatomy while simultaneously solving the image reconstruction problem. We do not claim novelty for bilinear parameterizations of unknown sensing operators as such; rather, the contribution is their adaptation to atlas-induced Jacobian uncertainty in neonatal DOT. We simulated Jacobians in 215 neonatal head models that were linearly registered to one reference head shape and performed PCA in the formed Jacobian atlas to construct basis matrices for the forward-operator variation. The bilinear tensor form was used to include the basis coefficients as secondary unknowns in the image reconstruction.

We implemented four approaches for improving the atlas mean Jacobian: the G--N algorithm, BCD, a two-block Gibbs sampler, and a version of the BAE method. In most cases, all algorithms improved the initial reconstruction both visually and by increasing the CNR and lowering the RMSE. However, the coefficients of the basis matrices could not be restored systematically. Finding a global MAP estimate for the joint or $y$-marginalized posterior is an interesting topic for further research.

\section*{Acknowledgements}
 We thank the anonymous referees for their valuable comments that considerably improved this manuscript. We acknowledge the computational resources provided by the Aalto Science-IT project. Ilkka Nissil{\"a} provided the formulas for amplitude and phase noise, which is highly appreciated. We thank Antti Hannukainen for guidance with the Kabsch algorithm for registration. We utilized the head model data made available from the Developing Human Connectome Project funded by the European Research Council under the European Union's Seventh Framework Programme (FP/2007-2013)/ERC Grant Agreement no.\ (319456). We thank Liam Collins-Jones for help with accessing the database. 

\appendix
\section{Proof of Proposition~\ref{prop:marg_post}}
\label{sec:marg_post}

Consider the posterior for $(y,x)$ from \eqref{eq:xy_posterior}. A direct calculation yields
\begin{align}
\label{eq:G_prod}
        \pi (y,x \mid b) &\propto 
        \exp \!{\Big( - \dfrac{1}{2} \big( \lVert b - A_0x - \Acal \bigcdot (y,x)_{2,3}  \rVert^{2}_{\Gamma_1^{-1}} +\lVert y \rVert^{2}_{\Gamma_2^{-1}} + \lVert x \rVert^{2}_{\Gamma_3^{-1}}   \big) \Big)} \nonumber \\[1mm]
        &= 
        \exp\!{\Big( - \dfrac{1}{2}  ( y - \hat{y})^{\top} \Gamma_{y | b,x}^{-1} ( y - \hat{y}) \Big)} \exp\!{\Big( - \dfrac{1}{2}  \big( x^{\top} \Gamma_3^{-1} x + \xi^{\top} \Gamma_1^{-1} \xi - c \big) \Big)} \nonumber \\[1mm]
        &=: f(y,x) \, g(b,x), 
\end{align}
where
\begin{align}
\Gamma_{y | b,x} &= \big(\Gamma_2^{-1} + A_{x,3}^{\top} \Gamma_1^{-1} A_{x,3} \big)^{-1}, \nonumber \\[1mm]
\xi &= b - A_0x, \label{eq:xi}\\[1mm]
\hat{y} &= (\Gamma_2^{-1} + A_{x,3}^{\top}\Gamma_1^{-1}A_{x,3})^{-1} A_{x,3}^{\top} \Gamma_1^{-1} \xi, \nonumber \\[1mm]
c &= \xi^{\top} \Gamma_1^{-1} A_{x,3} \Gamma_{y | b,x} A_{x,3}^{\top} \Gamma_1^{-1} \xi \label{eq:c}
\end{align}
are independent of $y$. Note that $\Gamma_{y | b,x}$ is the covariance matrix of $Y|b,x$,~i.e.,~of $Y$ conditioned on $B = b$ and $X = x$, by virtue of the linear relationship $b = A_0 x  + A_{x,3} Y + E$ and \cite[Theorem~3.7 \& Eq.~(3.16)]{Kaipio2006}.

The first term $f(y,x)$ on the right-hand side of \eqref{eq:G_prod} is almost Gaussian in $y$, i.e., Gaussian modulo an $x$-dependent multiplicative constant, which enables marginalizing the posterior with respect to $y$ in a closed form:
\begin{align} 
\label{eq:marg_post0}
    \pi( x \mid  b ) &=  \int_{\mathbb{R}^{p}} \pi( y, x \mid  b) \, {\rm d}y 
    \, \propto g(b, x) \int_{\mathbb{R}^{p}} f(y,x) \, {\rm d}y \propto | \Gamma_{y | b, x} |^{\frac{1}{2}} \, g(b, x).
\end{align}
Moreover, inserting \eqref{eq:c} in the definition of $g(b,x)$ gives
\begin{align*}
   - 2 \log g( b,x ) 
    &=  x^{\top}\Gamma_3^{-1}x  + \xi^{\top} \big( \Gamma_1^{-1} - \Gamma_1^{-1}A_{x,3}\big(\Gamma_2^{-1} + A_{x,3}^{\top} \Gamma_1^{-1} A_{x,3}\big)^{-1}A_{x,3}^{\top}\Gamma_1^{-1} \big) \xi   \\[1mm]
    &= x^{\top}\Gamma_3^{-1}x + (b - A_0 x)^{\top} \big( \Gamma_1 + A_{x,3} \Gamma_2 A_{x,3}^{\top} \big)^{-1}(b - A_0 x)  ,
\end{align*}
where the second step follows from the Woodbury matrix identity and \eqref{eq:xi}. Note that $\Gamma_{b | x} = \Gamma_1 + A_{x,3} \Gamma_2 A_{x,3}^{\top}$ is the covariance matrix of $B|x$ due to the linear relation $B = A_0 x + A_{x,3} Y + E$ for $X = x$ and the assumed mutual independence of $Y$ and $E$. In particular, $g( b,x )$ coincides with the exponential term on the right-hand side of \eqref{eq:marg_post}. Combined with \eqref{eq:marg_post0}, this proves the claim.

% Bibliography using bibtex %%%%%%%%%%%%%%%%%%%
\bibliographystyle{siamplain}
\bibliography{biblio.bib}

\end{document}